\documentclass[a4paper,11pt]{amsart}
\usepackage{amsmath,amsthm}
\usepackage{graphicx}
\usepackage{stmaryrd}
\usepackage{wrapfig}
\usepackage[left=2.5cm,right=2.5cm,textheight=22cm]{geometry}
\usepackage{tikz}
\usetikzlibrary{arrows.meta}
\usepackage{amsbsy}
\usepackage[utf8x]{inputenc}
\usepackage{mathrsfs}
\usepackage{version, tabularx, multicol}
\usepackage{graphicx,float,psfrag}
\usepackage[margin=5pt,justification=centering,labelformat=simple,labelfont=sc]{subcaption}

\usepackage{epstopdf}
\usepackage{color,latexsym,amsfonts}
\usepackage{mathtools}

\usepackage{pifont} 
\usepackage[inline]{enumitem}

\usepackage{hyperref}

\usepackage[draft]{fixme} 
\FXRegisterAuthor{jf}{ajf}{JFD}

\fxusetheme{colorsig}

\title{Site Frequency Spectrum in stationary branching populations}

\renewcommand{\Pr}{\mathbb{P}}
\newcommand{\Es}{\mathbb{E}}

\newcommand{\cov}{\mathrm{Cov}}

\newcommand{\diff}{\mathop{}\!\mathrm{d}}

\newcommand{\zm}{\zeta^\star}

\newcommand{\ind}{\mathbf{1}}
\newcommand{\Ed}{E_\mathrm{d}}
\newcommand{\Eg}{E_\mathrm{g}}

\newcommand{\ca}{{\mathcal A}}

\newcommand{\cf}{{\mathcal F}}

\newcommand{\cy}{{\mathcal Y}}
\newcommand{\cz}{{\mathcal Z}}

\newcommand{\cll}{{\mathcal L}}

\newcommand{\cn}{{\mathcal N}}
\newcommand{\cm}{{\mathcal M}}

\newcommand{\ct}{{\mathcal T}}
\newcommand{\cti}{{\ct^\mathrm{st}}}
\newcommand{\ctk}{{\ct^\mathrm{Kesten}}}
\newcommand{\ctcl}{{\ct^\mathrm{st}_\mathrm{clonal}}}

\newcommand{\Pt}{{\Phi}}
\newcommand{\cu}{{\mathcal U}}
\newcommand{\cx}{{\mathcal X}}

\newcommand{\E}{{\mathbb E}}

\newcommand{\N}{{\mathbb N}}

\renewcommand{\P}{{\mathbb P}}

\newcommand{\Q}{{\mathbb Q}}
\newcommand{\R}{{\mathbb R}}

\newcommand{\T}{{\mathbb T}}

\newcommand{\rN}{{\mathbf{\mathrm N}}}

\newcommand{\rd}{{\rm d}}

\newcommand{\bt}{{\mathbf t}}

\newcommand{\bZ}{{\mathbf Z}}
\newcommand{\bY}{{\mathbf Y}}

\newcommand{\fT}{\mathfrak{T}}

\renewcommand{\root}{{\varrho}}

\newcommand{\length}{{\mathscr L}}

\newcommand{\zc}{Z_\mathrm{cl}}

\newcommand{\Card}{{\mathrm {Card}}\;}

\newcommand{\inv}[1]{\mathop{\frac{1}{ #1}}\nolimits}
\newcommand{\expp}[1]{\mathop {\mathrm{e}^{ #1}}}

\theoremstyle{plain}
\newtheorem{theo}{Theorem}[section]
\newtheorem*{theo*}{Theorem}

\newtheorem*{cor*}{Corollary}

\newtheorem{lem}[theo]{Lemma}

\newtheorem{lemma}[theo]{Lemma}
\newtheorem{theorem}[theo]{Theorem}
\theoremstyle{remark}
\newtheorem{rem}[theo]{Remark}

\begin{document}

\author{Romain Abraham}
\address{Romain Abraham, Institut Denis Poisson,
Universit\'{e} d'Orl\'{e}ans,
Universit\'e de Tours, CNRS, France}
\email{romain.abraham@univ-orleans.fr}

\author{Jean-François Delmas}
\address{Jean-Fran\c{c}ois Delmas,  CERMICS, ENPC, Institut Polytechnique de Paris, CNRS, Marne-la-vall\'ee, France}
\email{jean-francois.delmas@enpc.fr}

\author{Patrick Hoscheit}
\address{Patrick Hoscheit, INRAE, MaIAGE, Université Paris-Saclay, 78350 Jouy-en-Josas, France}
\email{patrick.hoscheit@inrae.fr}

\date{\today}

\begin{abstract}
 This paper explores the Site Frequency Spectrum (SFS) in stationary 
 branching populations. We derive 
 estimates for the SFS associated with a sample from a continuous-state 
 branching process conditioned to never go extinct, utilizing a quadratic 
 branching mechanism. The genealogy of such processes is represented by a 
 real tree with a semi-infinite branch, and we compute the expectation 
 of the SFS under the infinitely-many-sites assumption as the 
 sample size approaches infinity. Additionally, we present a continuum 
 version of the SFS as a random point measure on the positive real line and 
 compute the density of its expected measure explicitly. Finally, we derive estimates 
 for the size of the clonal subpopulation carrying the same genotype as the most 
 recent common ancestor of the whole population at a given time.
\end{abstract}

\maketitle

\section{Introduction}

The Site Frequency Spectrum (SFS) of a genetic sample is a summary statistic of 
the full alignment that characterizes each mutation found in the sample by the 
number of individuals carrying it. It has been shown to reflect many features of
the past dynamics of the population from which the sample was taken, including 
variations in ancestral population size, selection, or the existence of 
population structure. In this paper, we will give estimates for the expectation
of the SFS associated to a sample from a continuous-state branching process 
conditioned never to go extinct. Such processes, first described in 
\cite{Lambert2007}, represent, at any given time, the size of an infinite stationary 
population undergoing branching. In the general case, the branching mechanism
is described by the Laplace exponent of a spectrally positive Lévy process. 
Here, we will focus on the quadratic case in which the underlying Lévy process is 
a Brownian motion with positive drift. The genealogy of such stationary branching 
processes can be represented by a metric space \((\ct,d)\) which is a 
\emph{real tree} with an infinite branch (see \cite{Evans} for a comprehensive introduction 
to real trees and their use in probability theory). We will use the distribution of 
the subtree spanned by a uniform sample of \(n\) leaves at a given time, which 
was given in \cite{abrahamExactSimulationGenealogical2021}, to compute the 
expectation of the SFS of 
such a sample (Theorem \ref{the:SFS_discret}), under the infinitely-many-sites 
assumption, as \(n\) goes to 
\(\infty\). We also present a continuum version of the SFS as a random 
point measure on \(\mathbb{R}_+\), following the framework introduced in 
\cite{Duchamps2018}. We show that the expected measure has a density with 
respect to Lebesgue measure, which we compute explicitly (Theorem \ref{the:meanSFS}). 
Finally, we study the part of the population at a given time that carries no  
additional mutations compared to its most recent common ancestor. We compute the
expectation of the size of this subpopulation, as well as the ratio between its size 
and the size of the whole extant population (Theorem \ref{th:moment-clone}).

We will now review some results from the literature on the SFS. 
Given a rooted real tree \(\ct\) with \(n\) leaves, let 
\(\mathcal{M}\) be an independent Poisson point process on \(\ct\) with 
intensity \(\mu>0\). Each atom of \(\mathcal{M}\) is a mutation that is carried 
by the whole subpopulation descended from it. Each mutation occurs at a 
different locus (\emph{i.e.} we assume the \emph{infinitely-many-sites model}), 
and we assume that the ancestral allele at each locus is known. Thus we can 
define the SFS of a $n$-sample as the vector:
\[
 \xi^{(n)}=(\xi_1^{(n)},\dots,\xi_{n-1}^{(n)}),
\] 
where \(\xi_k^{(n)}\) is the number of mutations carried by exactly \(k\) 
individuals in the sample.

When \(\ct\) is a Kingman coalescent tree with \(n\) leaves, the first 
moment of the SFS can be explicitly computed:
\begin{equation}\label{eq:SFSKingman}
 \Es[\xi_k^{(n)}] = \frac{\theta}{k},\quad k=1,\dots,n-1,
\end{equation}
where \(\theta\) is the population-scaled mutation rate \(\theta=4N_e\mu\). In 
this expression, \(N_e\) is the effective population size parameter and \(\mu\)
as above, the per-lineage mutation rate. See \cite{Fu1995} for a derivation of 
this expression, as well as results on second moments. This result was 
subsequently extended to accommodate relaxations of the strict assumptions 
underlying the Kingman coalescent. Notably, Griffiths and Tavaré \cite{Griffiths1998}
established the following formula for the expectation of the SFS:
\begin{equation}\label{eq:SFSSomme}
 \Es[\xi_k^{(n)}] = \frac{\theta}{2} \sum_{i=2}^{n-k+1} i p^{(n)}(i,k) 
 \Es[T^{(n)}_i],
\end{equation}
where \(p^{(n)}(i,k)\) is the probability that at the time the coalescent 
has \(i\) blocks, a given one of them contains exactly \(k\) leaves, and where
\(T^{(n)}_i\) is the amount of time when the coalescent has 
exactly $i$ blocks. This formula holds for variable population sizes, but the 
expectations might not be explicitly computable.

Equation \eqref{eq:SFSSomme} can be generalized to the case of 
\(\Lambda\)-coalescents \cite{Birkner2013}, or even \(\Xi\)-coalescents 
\cite{Blath2016}. Asymptotic results for \(\Lambda\)-coalescents in the case
where \(\Lambda\) is regularly varying at 0 with index \(1<\alpha<2\) (meaning 
that \(\Lambda(\diff x)=f(x)\diff x\) with \(f(x)\sim Ax^{1-\alpha}\) as \(x\to 0\)) are 
found in \cite{Berestycki2014c}: 
\[
 \lim_{n\to \infty}n^{\alpha-2}\xi_k^{(n)} = \frac{\theta}{2} C_{A,\alpha} 
 \frac{(2-\alpha)\Gamma(k+\alpha-2)}{k! \Gamma(\alpha-1)}, 
\]
almost surely for fixed \(k\ge 1\), where the constant \(C_{A,\alpha}\) is 
explicit. Recently, Kersting et al. \cite{Kersting2019} were able to obtain a 
closed integral formula for the SFS in the special case of the 
Bolthausen-Sznitman coalescent:
\[
 \Es[\xi_k^{(n)}] = \theta n \int_0^1 \frac{\Gamma(k-p)}{\Gamma(k+1)} 
 \frac{\Gamma(n-k+p)}{\Gamma(n-k+1)} \frac{dp}{\Gamma(1-p)\Gamma(1+p)},
\]
which leads to the following asymptotics for large values of \(n\):
\begin{equation}\label{eq:SFSBS}
 \Es[\xi_k^{(n)}] \sim \begin{cases} 
 \frac{\theta n}{\log n}& \text{ if }k=1,\\
 \frac{\theta n}{k(k-1)} \frac{1}{\log^2(n/k)}&\text{ if }
 k\ge 2,\ k/n\to 0, \\
 \frac{\theta}{n} f_1(u)
 &\text{ if }k/n\to u\in (0,1),\\
 \frac{\theta}{n-k}\frac{1}{\log(n/(n-k))}&\text{ if }1-k/n\to 0,
 \end{cases}
\end{equation}
where \(f_1(u)=\int_0^1 u^{-1-p}(1-u)^{p-1}\sin(\pi p)/(\pi p)dp\) is the 
asymptotic profile of the SFS. 

Another vein of research has focused on the use of the SFS to 
infer parameters of the coalescent, such as the \(\Lambda\) or \(\Xi\) measure
for exchangeable coalescents or ancestral demographic fluctuations when 
effective population size is not assumed to be constant through time. Starting
with \cite{Myers2008}, several negative and positive identifiability results 
have been proven, see \cite{Bhaskar2014a,Kim2015,Koskela2013,Terhorst2015,miropinaEstimatingLambdaMeasure2023},
that put 
the theory of ancestral demographic reconstruction on solid statistical footing. 
This has also led to numerical methods to efficiently compute the SFS under a given 
coalescent model and with a given demography \cite{Spence2016} and to use them for 
hypothesis testing \cite{Birkner2014,Koskela2017} or parameter inference 
\cite{Koskela2015,Demography2018}. 

The study of site frequency spectra in branching processes, which is the 
subject of the present paper, is facilitated by the description of the 
genealogy of extant populations using \emph{coalescent point processes} (CPP), 
starting with \cite{Popovic2004}. CPPs are random genealogies defined using 
sequences of random variables \((H_i,\ i\ge 1)\) representing the time to the most recent common ancestor (TMRCA) of consecutive individuals:
\[
 \mathrm{TMRCA}(i,i+1) = H_{i+1},\quad i\ge 1.
\]
Lambert \cite{Lambert2009} proved that for an
independent CPP, under mild moment conditions and assuming uniform mutations 
along lineages with rate \(\theta\), the following holds for fixed $k\ge 1$:
\[
 \lim_{n\to \infty} \frac{\xi_n^{(k)}}{n} = \theta \int_0^\infty 
 \frac{1}{W(x)^2}\left(1-\frac{1}{W(x)}\right)^{k-1}dx \quad \text{a.s.,}
\]
where \(W(x)=1/\Pr(H>x)\) is the \emph{scale function} of the random variables
underlying the CPP. This result was later extended to more general mutation 
distributions in \cite{Duchamps2018}. 

Most recently, Schweinsberg and Shuai 
\cite{schweinsbergAsymptoticsSiteFrequency2023} have examined 
site frequency spectra for critical or supercritical birth and death processes, 
using CPP representations of the genealogy of \(n\) sampled individuals 
at a given time \(T_n\to \infty\), due to \cite{Lambert2018a}. In the critical 
case, they found Kingman-like asymptotics for the total length 
of branches with exactly \(k\) sampled leaves in their descendance, and proved
asymptotic normality.

In this work, we will study the SFS associated to a neutral, time-homogeneous
mutation process at rate \(\mu\) in populations modelled by a stationary continuous-state branching process $(Z_t, t\in \R)$, for quadratic 
branching mechanisms given by:
\[
 \psi(u)=\beta u^2 + 2\beta\theta u,
\]
where $\beta>0$ is a time scaling parameter and $1/\theta=\Es[Z_t]$ for any $t\in \R$ can be seen as a population size scaling parameter. In such a population, sampling \(n\) individuals at time 0, representing the present, our first result (Theorem \ref{the:ExpectedSFS}) gives an exact formula for the expected SFS.

\begin{theo*}[Expected site frequency spectrum]
For $1\leq k\leq n-3$ we have:
\begin{multline}
    \beta  \theta \E[\xi^{(n)}_k]= \frac{1}{2k} \left( 1 + \frac{n}{n-k-2}\right)
    + \frac{n}{(n-k)(n-k-2)} \\
    + \frac{n(n-1)}{(n-k)(n-k-1)(n-k-2)}\, \left( \Psi(k) - \Psi(n-1)\right),
\end{multline}
where $\Psi(u)=\Gamma'(u)/\Gamma(u)$ is the digamma function, and for $k=n-2, n-1$:
\[
\beta  \theta   \E[\xi^{(n)}_{k_n}]=\frac{2}{3n}+ O\left(\frac{1}{n^2}\right).
\]
\end{theo*}
It is not difficult to derive from this result that for $k_n/n \rightarrow u\in [0,1]$, with $1\leq k_n \leq n-1$:
\begin{equation}\label{eq:AsympProfile}
\lim_{n\rightarrow \infty} \beta  \theta  k_n \E[\xi^{(n)}_{k_n}]=\mathbf{g}(u)
\quad\text{with}\quad
\mathbf{g}(u)=
\frac{1}{2} \left(1+ \frac{1}{1-u} \right)
+ \frac{u}{(1-u)^2}+
\frac{u\log(u)}{(1-u)^3} \cdot
\end{equation}
The function $\mathbf{g}$ is the asymptotic profile of the SFS, and should be compared to similar functions in the Kingman case (for which $\mathbf{g}(u)=1$) and the Bolthausen-Sznitman case (see Equation \eqref{eq:SFSBS}), where it exhibits its distinctive U-shape. It is plotted in Fig. \ref{fig:expected_sfs}. It is elementary to check that:
\begin{equation}
    \label{eq:def-bf:g}
\mathbf{g}(u)=
\begin{cases}
    1+ u \log(u) + O(u) & \text{as $u\rightarrow 0+$,}\\
    \frac{2}{3} + O(1-u) &
\text{as $u\rightarrow 1-$.}
\end{cases}
\end{equation}

\begin{figure}
    \centering
    \includegraphics[width=.7\textwidth]{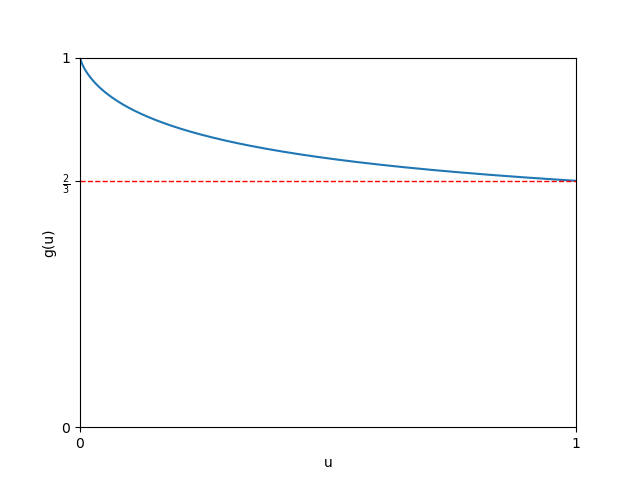}
    \caption{Asymptotic profile of the expected site frequency spectrum $\mathbf{g}(u)$ for $0\le u \le 1$.}
    \label{fig:expected_sfs}
\end{figure}

We can also give a more precise result (Theorem \ref{the:SFS_discret}), describing the conditional expectation of the 
SFS given the population size $Z_0$.
\begin{theo*}[Conditional expectation of the SFS]
For $1\le k\le n-1$, we have:
\begin{equation} \label{eq:SFSdiscrete}
   \frac{\beta}{\mu Z_0}\,    \E[\xi_k^{(n)}  |\, Z_0] = \inv{k} +
   \inv{k} g_1\left(\theta Z_0, 
      \frac{k}{n}\right)+  \frac{\sqrt{k} }{n^2}\, g_2\left(\theta Z_0,
      \frac{k}{n}, n\right),
\end{equation}
where the function $g_1$ is explicit (see Equation~\eqref{def:g1}) and represented in Fig.\ref{fig:g1_plot}, and where $g_2$ is uniformly bounded (thus the third term of the sum vanishes in the limit $n\to \infty$). 
\end{theo*}

The expansion~\eqref{eq:DL-gn} gives that for $\lim_{n\rightarrow \infty} k_n/n=u $ with $u$ close to $0$:
\[
 \lim_{n\rightarrow \infty}  k_n \, \E[\xi_{k_n}^{(n)}  |\, Z_0] \simeq 
\frac{\mu Z_0}{\beta}
\big(1+ 2(\theta Z_0-1) u\log(u)\big),
  \]
and the expectation w.r.t.\ $Z_0$ of the above right hand-side is exactly $\mu/(\beta \theta)$ times $1+ u\, \log(u)$, which is the expansion of $\mathbf{g}$ around $0$. Hence, we can interpret the value $\mathbf{g}(0)=1$ as the population size $\Es[Z_0]$ at time 0, rescaled by its expectation $1/\theta$. Going further, the value $\mathbf{g}(1)=2/3$ is reminiscent of the relative expected population size $\Es[Z_A]/\Es[Z_0]$ at the time $A$ of the most recent common ancestor that was computed in \cite{chenSmallerPopulationSize2012}. Thus, it would be interesting to compare $\mathbf{g}(u)$ with the relative expected population  size at time $uA$, that is, $\Es[Z_{uA}]/\Es[Z_{0}]$ for $ 0\le u\le 1$. 

The proof of Theorem \ref{the:SFS_discret} relies mostly on a representation theorem of the genealogy of a sample of 
\(n\) individuals using a construction similar to the CPP, obtained in 
\cite{abrahamExactSimulationGenealogical2021}. This construction is similar to the coalescent point processes (CPPs) 
extensively studied by Lambert since their introduction in 
\cite{Lambert2009}. Most notably, in \cite{Duchamps2018}, the authors 
consider general CPPs associated to Poisson processes and compute asymptotic features of the SFS for these trees. The stationary setting 
used in the present work breaks the Poissonian structure of the coalescent point process
and leads to the more involved construction of 
\cite{abrahamExactSimulationGenealogical2021}. Note also that in that paper, Abraham and Delmas studied the 
asymptotics of the total length $\Lambda_n$ of the genealogy $\ct_n$ of a $n$-sample of individuals sampled uniformly in the extant population at time 0, proving in particular (Theorem 
5.1 in \cite{abrahamExactSimulationGenealogical2021}) that:
\begin{equation*}
    \Es[\Lambda_n|Z_0] = \frac{Z_0}{\beta} \log\left( \frac{n}{2\theta Z_0} \right) + O(n^{-1}\log(n)).
\end{equation*}
In our context, $\Lambda_n$ is related to the number of segregating sites $S_n$ defined by $S_n = \sum_{k=1}^{n-1}
\xi_k^{(n)}$, since, conditionally on $\ct_n$, $S_n$ is a Poisson random variable with mean $\mu\Lambda_n$. In this sense, Theorem \ref{the:SFS_discret} can be seen as an 
unfolding of the asymptotics of $S_n$ given in \cite{abrahamExactSimulationGenealogical2021}. Indeed,  we 
recover that $\E[\Lambda_n|Z_0]\simeq (Z_0/\beta) \log(n)$ by summing \eqref{eq:SFSdiscrete} over $1\le k \le n-1$.

We will also present a version of the SFS defined directly on the continuum stationary tree $\cti$, in the spirit of the 
construction of \cite{Duchamps2018}. In Theorem \ref{the:meanSFS}, we compute the expected intensity of the
continuum SFS measure, defined by $\Pt = \sum_{i\in I_M} \delta_{\cz_0(\cti(m_i))}$, where $\cz_0(\cti(m_i))$ is the size at time 0 of the population carrying a given mutation $m_i$. Let $\Gamma(a,b)$ be the incomplete Gamma function.

\begin{theo*}[Expected density of the continuous SFS measure]
The mean site frequency measure \(\Es[\Phi(\rd r)]\) of the stationary tree
  $\cti $ is absolutely continuous with respect to the Lebesgue measure on \(\mathbb{R}_+\), with
  density given by:
 \begin{equation}
 f(r) = \frac{\mu}{\beta} \left(\frac{\expp{-2\theta r}}{\theta r} + 
 \expp{-2\theta r}+2\theta r \, \Gamma(0,2\theta r)\right). 
 \end{equation}
\end{theo*}

We also give results (Theorem~\ref{th:moment-clone}) about the fraction of the population carrying 
the same alleles as its most recent common ancestor, called the \emph{clonal subpopulation}.

\begin{theo*}[Absolute and relative size of the clonal subpopulation]
If $Z_\mathrm{cl}$ is the absolute size of the clonal subpopulation, and $R=Z_\mathrm{cl}/Z_0$ its relative size, then we have: 
\begin{equation}
 \E[R]=\frac{2}{(\alpha+1)(\alpha+2)}
 \quad\text{and}\quad
 \E[\zc]=\E[R Z_0]= \frac{3}{\alpha+3}\, \E[R] \,\E[Z_0]
,
\end{equation}
where $\alpha=\mu/(2\beta\theta)$, and thus:
\[
    \cov(R,Z_0) = - \frac{\alpha}{\alpha+3} \,\E[R]\,  \E[Z_0] <0.
\]
\end{theo*}

An interesting extension of the present work would be to generalize Theorem 
\ref{the:SFS_discret} to branching mechanisms containing an infinite jump 
measure, such as stable branching mechanisms \(\psi(u)=u^\alpha,\ 1<\alpha<2\). 
This would require an ancestral construction similar to the CPP, but allowing 
for multiple branches coalescing at the same time. We expect that the SFS of 
such populations would not have a Kingman-like form, but possibly exhibit a 
U-shape, typical of genealogies described by \(\Lambda\)-coalescents 
\cite{freundInterpretingPervasiveObservation2023}. 

The rest of the paper is organized as follows: in Sections \ref{sec:Prelim} and \ref{sec:sample-T},
below, we will introduce the objects and notations used in the paper, in 
particular the constructions of \cite{abrahamExactSimulationGenealogical2021}.
In Section \ref{sec:DiscreteSFS}, we will give the proof of Theorem 
\ref{the:SFS_discret}, then, in Section \ref{sec:ContinuousSFS}, we will present
the continuum version of the SFS and compute the density of its expected 
measure. Finally, in Section \ref{sec:PopClonale}, we prove the results concerning the continuum SFS and the clonal subpopulation.

\section{Preliminaries}\label{sec:Prelim}

\subsection{Stationary continuous branching processes}

We consider a critical quadratic branching mechanism \(\psi(u)=\beta u^2\) and 
the associated family \((\psi_\theta,\ \theta>0)\) of sub-critical branching 
mechanisms:
\begin{equation}\label{eq:BranchingMechanism}
 \psi_\theta(u)=\psi(u+\theta)-\psi(\theta)=\beta u^2 + 2\beta\theta u. 
\end{equation}
Let $\theta>0$ be fixed. 
We denote by \(\Pr_x\) the distribution of a continuous-state branching (CB)
process \(\bY=(Y_t,\ t\ge 0)\) started at \(x >0\), with branching mechanism 
\(\psi_\theta\). The process \(\bY\) is the solution of the following Feller diffusion 
equation where $(B_t, \, t\geq 0)$ is a standard Brownian lmotion:
\[
  \diff Y_t = \sqrt{2\beta Y_t}\diff B_t - 2\beta\theta Y_t \diff t
  \quad\text{and}\quad Y_0=x.
\]
We also consider the associated canonical measure \(\rN\), 
which is a \(\sigma\)-finite measure on the space \(\mathbb{D}\) of nonnegative 
continuous functions \(f\) such that if \(f(s)=0\) for some \(s>0\), then 
\(f(t)=0\) for all \(t\ge s\). It can be thought of as the ``distribution'' of a continuous-state branching process
started at an infinitesimal population size conditioned not to die immediately, in the same way that the Itô excursion measure of Brownian motion 
represents the ``distribution'' of a single Brownian excursion outside $0$ (see \cite{LeGall1999} for more details on the 
relation between the excursion measure and the canonical measure of branching processes). By convention, we 
will write $\rN[X]=\int X(f)\, \rN(\rd f)$ for the integral with respect to $\rN$ of integrable functions $X:\mathbb{D}\to \R$. 
Of course, the function $X$ can also seen as a random variable under $\P_x(\rd f)$. 

The 
Laplace transform of the one-dimensional distributions of \(\bY\) is given by, for $\lambda\ge 0$ and $t\ge 0$:
\[
 \Es_x [\exp(-\lambda Y_t)] = \exp(-xu(t,\lambda)),
\]
where the positive function $u$, which solves the equation $u(t,\lambda)+ \int_0^t \psi_\theta(u(s, \lambda))\, \rd s=\lambda$, see \cite[Eq.(3.3)]{Li2011}, is given by: 
\[
 u(t,\lambda) = \rN[1-\exp(-\lambda Y_t)] = 
 \frac{2\theta \lambda}{(2\theta+\lambda)\exp(2\beta\theta t) - \lambda}\cdot
\]
In particular, we get \(\rN[Y_t]=\exp(-2\beta\theta t)\).
The tail distribution of the extinction time \(\zeta=\inf\{t>0,\ Y_t=0\}\) under the canonical measure is given by, for $t\ge 0$:
\begin{equation}\label{eq:LoiTempsExtinction}
 c(t) = \rN[\zeta > t]  =\lim_{\lambda \rightarrow \infty } u(t,
  \lambda)=\frac{2\theta}{\exp(2\beta\theta 
 t)-1}\cdot
\end{equation}
It is possible to construct a stationary version of this CB 
process using an immigration process. Let: 
\[
 \cn(\diff t,\diff Y) = \sum_{i\in I_Z} \delta_{(t_i,Y_i)}(\diff t,\diff Y)
\]
be a Poisson point measure on \(\mathbb{R}\times \mathbb{D}\) with intensity 
\(2\beta \diff t \, \rN[\diff Y]\). 
Following~\cite[Section~3.2]{chenSmallerPopulationSize2012}
on Q-processes, we define the stationary CB process \(\bZ=(Z_t, \, t\in 
\R)\) by:
\[
 Z_t = \sum_{t_i\le t} Y^i_{t-t_i}.
\]
This stationary CB process appears also as the limit of the CB process conditioned
not to be extinct \cite{Lambert2007}. It is distributed as the stationary Feller diffusion, which is a solution of the following equation (see Section 7.1 of \cite{chenSmallerPopulationSize2012}):
\[
 \diff Z_t = \sqrt{2\beta Z_t}\diff B_t +2\beta(1-\theta Z_t)\diff t,\quad t\in 
 \R,
\]
and the (stationary) one-dimensional marginal $Z_t$ is distributed 
as the sum of two independent exponential random variables with
parameter $2\theta$ (see  \cite[Eq.~(46)]{chenSmallerPopulationSize2012}). 
In particular we have $\E[Z_t]=1/\theta$.

\subsection{Genealogical tree of the stationary CB process}
\label{sec:def-cti}
The genealogy of the CB process $\bY$ (under the canonical measure) can
be described as a random tree encoded by a Brownian excursion as
follows. For a function \(g\in\mathbb{D}\) define a pseudo-distance on
$\R_+$ by, for $s,t\in \R_+$:
\[
 d_g(s,t) = g(s)+g(t)-2m_g(s,t)
 \quad\text{with}\quad
m_{g}(s,t)=\inf_{[s\wedge t,s\vee t]}g.
\]
The quotient metric space \(\ct_g=\R_+{/\{d_g=0\}}\), with the 
metric $d_g$, is then a \emph{real tree} \cite{Evans2005}; the equivalence class of
\(0\), denoted by $\root_g$, is called the root of the tree $\ct_g$.
The height of $x\in \ct_g$ is defined as its distance to the root, that
is, as $g(t)$ for any $t\in \R_+$ in the equivalence class $x$. We say that 
$s\in \R_+$ is an ancestor of $t\in \R_+$ if $g(s)=m_g(s,t)$; this defines a partial 
order on $\ct_g$, and we write $s\preceq t$ identifying $s$ and $t$ with their 
equivalence class.

In this framework, $g$ serves as the \emph{height process} (or contour process) that explores $\ct_g$ in a depth-first manner. For any $s \in \R_+$, an excursion of $g$ above the level $g(s)$ represents the entire sub-family of descendants originating from $s$. Formally, if $(s, t)$ is an interval such that $g(r) > g(s)$ for all $r \in (s,t)$ and $g(t)=g(s)$, the equivalence classes of all $u \in [s, t]$ constitute the sub-tree rooted at $s$, consisting of all individuals $x \in \ct_g$ such that $s \preceq x$. The return of $g$ to the level $g(s)$ at time $t$ signifies the point at which this entire sub-lineage has been explored.

Recall $\theta>0$ is given. Consider a Brownian motion with negative
drift
$B^{(\theta)}= (B^{(\theta)}_t=\sqrt{2/\beta}B'_t-2\theta t,\ t\ge 0)$,
where \(B'\) is a standard Brownian motion.
Let $\N[d\ct]$ denote the
push-forward measure of the Itô positive excursion measure of the
Brownian motion $B^{(\theta)}$ through the application
$ g \mapsto \ct_g$. The $\sigma$-finite measure $\N[d\ct]$ is defined on
the Polish space $\T$ of compact rooted real trees endowed with the so-called
Gromov-Hausdorff distance (where pointed compact metric spaces are identified up
to an isomorphic transformation). We simply denote the root of $\ct$ by
$\root$. 

We define a local time process of the tree $\ct$ denoted by $\cy=(\cy_a, 
a\geq 0)$ where $\cy_a$ is a random measure on $\ct$ which informally is the uniform measure on the elements of $\ct$ at distance $a$ from the root. 
More formally, let $(\ell^a(\rd s), a\geq 0)$ be the local time process of
$B^{(\theta)}$, with $\ell^a(\R_+)$ the total local time at level
$a$. For any $g\in \mathbb{D}$, let $\Pi_g$ be the natural projection from $\R_+$ on 
$\ct_g$. Then, for any $a\ge 0$, denote by $\cy_a$ the push-forward measure of 
$\ell^a$ on $\ct$ through the map $\Pi_{B^{(\theta)}}$.  According to
\cite[Theorem~1.4.1]{Duquesne2002}, the total mass process
$(\cy_a(\ct) ,\ a\ge 0)$ is distributed under $\N$ as the CB process $\bY$
with branching mechanism $\psi_\theta$ under the canonical measure
$\rN$. For this reason, we shall identify $Y_a$ with $\cy_a(\ct)$ for all
$a\geq 0$, and thus see the tree $\ct$ as the genealogical tree
associated to the CB $\bY$. See also~\cite{Duquesne2005a} for a direct
construction of measures $(\cy_a, a\geq 0)$ from the tree $\ct$. The
maximal height (distance from the root) of $\ct$ is distributed as the
lifetime $\zeta$ of the CB process $\bY$, and it will also be denoted by
$\zeta$.

\medskip

We now informally describe the genealogical tree associated to the
stationary CB process $\bZ$, see \cite{chenSmallerPopulationSize2012}.
Let $\sum_{i\in I_T} \delta_{(h_i,\ct_i)}(\rd h, \rd \bt)$ be a Poisson
point measure on \(\R\times \T\) with intensity
\(2\beta \diff h\N[\diff \bt]\). The tree \(\cti\) is
obtained by grafting the trees \(\ct_i\) at height \(h_i\) along the
infinite spine \(\R\) (and the root of $\ct_i$ is identified with
$h_i$ on the infinite spine $\R$). 
We refer to~\cite[Section~5.5]{abrahamBrownianContinuumRandom2022} for a more formal framework for rooted trees with a rooted distinguished (or marked) subtree; here the marked subtree is given by the spine $\R$ with root $0\in \R$. Let us mention that in this framework the grafting procedure is a  well defined measurable function. 

The local time process $(\cz_t, t\in \R)$
associated to \(\cti\) is the sum at each level
\(t\) of the local times at level \(t-h_i\) of all trees \(\ct_i\) with
\(h_i\leq t\): $\cz_t=\sum_{h_i\leq t} \cy^{i}_{t- h_i}$, where $\cy^i$
is the local time process of the tree $\ct_i$.
Then, the total mass process $(\cz_t(\cti), t\in \R)$
is distributed as the stationary CB process $\bZ$ with 
branching mechanism \(\psi_\theta\). As above, we shall identify
$\cz_t(\cti)$ with $Z_t$. Notice that the measure $\cz_t$ puts
mass only on the set of leaves of $\cti$ at level $t$. 

\subsection{Quantities related to the genealogical tree}
The height of
$x\in \cti$, say $H(x)$, is defined as $x$ if $x$ belongs to the
infinite spine $\R$ or, if $x$ belongs to the $\ct_i$ grafted at height
$h_i$, as its height in $\ct_i$ plus $h_i$. We define a partial order
on $\cti$ by $x\preceq y$ for $x,y\in \cti$ if either 
\begin{enumerate*}[label=(\roman*)]
\item $x$ and $y$ belong to the infinite spine $\R$ and $H(x)\leq H(y)$, or
\item $x$ belongs to the infinite spine $\R$ and $y$ to the tree $\ct_i$ grafted at
level $h_i$ with $H(x)\leq h_i$, or
\item $x$ and $y$ belong to the same tree $\ct_i$ and $x$ is an ancestor of $y$ in 
$\ct_i$.
\end{enumerate*}
For $x\preceq y$ we define $\llbracket x,y \rrbracket=\{z\in \cti\,
\colon\, x \preceq z \preceq y\}$ the branch from $x$ to $y$. It can be
isometrically identified with the segment $[H(x), H(y)]$ of $\R$. The
length measure $\length(\rd x)$ on $\cti$ is defined through its restriction to
$\llbracket x,y \rrbracket$ for all $x\preceq y$ as the image of the
Lebesgue measure on $[H(x), H(y)]$.

For a set $(x_j,\ j\in J)$ of
elements of $\cti$, we define the set of its ancestors as
$\{x\in \cti\, \colon\, x \preceq x_j \quad\text{for all}\quad j\in
J\}$. If this set is not empty, then it has a maximal element (for the
partial order $\preceq$) which is called the most recent common ancestor
(MRCA) of $(x_j, j\in J)$ and its height is the time to the MRCA (TMRCA).
We shall consider the time $-A$ of the MRCA of the extant population at
time $0$. Denoting by $\zeta_i$ for the maximal height of the tree
$\ct_i$, it is also defined as:
\[
 A=- \min\{h_i\, \colon\, \zeta_i+ h_i\geq 0\}.
\]

We also define $N_t$ as the number of ancestors at time $-t$ of the extant population
living at time 0 minus 1 (that is, we don't take into account the
infinite spine):
\[
  N_t=\Card \{i\in I_T\, \colon h_i<-t \quad\text{and}\quad  \zeta_i+
  h_i\geq 0\}.
\]
In particular, we have that a.s.:
\begin{equation}
  \label{eq:A-Nt}
  \{A>t\}=\{N_t\geq 1\}. 
\end{equation} According
to~\cite{chenSmallerPopulationSize2012}, we have that $N_t$ is,
conditionally on $Z_{-t}$, 
distributed as a Poisson random variable with mean $c(t) Z_{-t}$. In
particular, we have:
\begin{equation}
  \label{eq:ENt}
  \E[N_t]=\frac{c(t)}{\theta}\cdot
\end{equation}
   
\subsection{The Kesten tree}

We shall also use the so-called  Kesten tree $\ctk$ which is obtained by
grafting the trees \(\ct_i\) at height \(h_i>0\) along the semi-infinite
spine \(\R_+\) rooted  at $\root=0\in \R_+$.   The local  time
process  $(\cz^\mathrm{Kesten}_t, t\in  \R)$ associated  to \(\ctk\)  is
then                             defined                             as:
$\cz^\mathrm{Kesten}_t=\sum_{0<h_i\leq   t}  \cy^{i}_{t-   h_i}$.  Then, as a direct consequence of~\cite[Theorem~4.5]{Duquesne2005a} stated in the more general framework of L\'evy trees, the
 total
mass process $Z^\mathrm{Kesten}_{(0,t]}=(Z^\mathrm{Kesten}_s=\cz^\mathrm{Kesten}_s(\ctk), s\in (0,t])$ up to time $t$
is distributed as  the size-biased distribution of $Y_{(0,t]}=(Y_s, s\in (0, t])$ under the
excursion measure, that is, for $t>0$ and $h$ a measurable non-negative function:
\begin{equation}
  \label{eq:EZK}
  \E\left[h(Z^\mathrm{Kesten}_{(0,t]})\right]=\frac{\rN[Y_t\,
    h(Y_{(0,t]})]}{\rN[Y_t]}=\expp{2\beta \theta t}\rN\left[Y_t
\,    h(Y_{(0,t]})\right]  . 
\end{equation}


\section{Coalescent Point Process of sampled stationary trees}
\label{sec:sample-T}
We recall the following construction from
\cite{abrahamExactSimulationGenealogical2021}. Let $\cti$ be the
genealogical tree associated to the stationary CB process $\bZ$ defined
in the previous section. We shall consider the genealogical sub-tree
$\ct_n$ spanned by $n$ individuals uniformly chosen among the population
at time $0$. More precisely, let \((\cx_k,\ k\in \N^*)\) be independent
leaves of $\cti$ at a given level, say $0$ for simplicity, chosen
uniformly, that is according to the probability measure $\cz_0/Z_0$. For
$n\in \N^*$, let $\ct_n$ be the subtree spanned by the leaves
\(\cx_1, \ldots, \cx_n\) (that is, the smallest subtree of \(\cti\)
containing \(\cx_1,\ldots,\cx_n\)) rooted at the MRCA of
\(\cx_1, \ldots, \cx_n\). We refer to
\cite{abrahamExactSimulationGenealogical2021} for a more formal
definition. We now give an elementary representation of the tree
$\ct_n$.

\begin{enumerate}[label=(\roman*)]
\item Let \((\Eg, \Ed)\) be independent exponential random variables with 
parameter \(2\theta\); so that $Z_0$ is distributed as $\Eg+\Ed$. For
simplicity, we identify $Z_0$ with $\Eg+\Ed$.
 Let also \((U_k,\ k\in \N^*)\) be independent 
random variables, uniformly distributed on \([0,1]\), independent of 
\(\Eg,\Ed\). We define the positions $X_0=0$ and \(X_k = Z_0 U_k - \Eg\)
for \(k\in \N^*\). The position $X_0$ corresponds to the individual alive
at time $0$ of the immortal lineage.

\item Let $n\in \N^*$ be fixed. We consider the set of ``leaves'' 
 \(\mathcal{L}_n = \{-\Eg,\Ed,X_0,\dots,X_{n-1}\}\) and the corresponding
 order statistics $X_{(0)}=-\Eg<X_{(1)}< \ldots <X_{(n)}< X_{(n+1)}=\Ed$.
 For $k\in \{0, \ldots, n+1\}$ we consider the interval 
 $[X_{(k)}, X_{(k+1)}]$ for $X_{(k)}<0$, 
$[X_{(k-1)}, X_{(k)}]$ for $X_{(k)}>0$, and the singleton $\{X_{(k)}\}$
 for $X_{(k)}=0$, and denote by $I_k$ its length. Notice that
 $\sum_{k=0}^{n+1} I_k=Z_0$.

\item Recall the function $c$ defined in~\eqref{eq:LoiTempsExtinction}. 
For $\delta>0$, let $\zeta^*(\delta)$ be a random variable on \((0,\infty)\)
whose distribution is given by:
\[
 \P(\zeta^*(\delta)\leq t)= \expp{- \delta c(t)}
 \quad\text{for}\quad t>0.
\]
In particular, $\zeta^*(\delta)$ is distributed as:
\begin{equation}
 \label{eq:law-z*}
 \inv{2\beta\theta}\, \log\left(1+ \frac{2\theta \delta}{E}\right),
\end{equation}
where $E$ is an exponential random variable with mean 1. Notice that if
$(\zeta_i, \, i\in I_Z)$, with $I_Z$ at most countable, are independent
random variables with $\zeta_i$ distributed as $\zeta^*(\delta_i)$, then
$\sup_{i\in I_Z} \zeta_i$ is distributed as $\zeta^*(\delta)$ with
$\delta=\sum_{i\in I_Z} \delta_i$.

\medskip

Conditionally on 
\(\mathcal{L}_n\), let \((\zeta_k,\ 0\le k\le n+1)\) be 
independent random variables such that $\zeta_k$ is distributed as
$\zeta^*(I_k)$, with $E$ in~\eqref{eq:law-z*} independent of $I_k$,  for $0\leq k\leq n+1$, and consider the ancestral point measure on
$ \R\times \R_+$ (notice the sum is from 1 to $n$):
\begin{equation}
 \label{eq:def-PP-A}
 \ca_n=\sum_{k=1}^n\delta_{(X_{(k)},\zeta_k)}.
\end{equation}
Notice that $(0,0)$ is an atom of $\ca_n$.

 Finally, let \(\fT _n\) be the 
ancestral tree associated defined as following: attach the semi-infinite branch \((-\infty,0]\) at
the position $X_0=0$ on the segment \([-\Eg,\Ed]\), and for all
\(1\le k\le n\), such that $X_{(k)}\neq 0$, attach a branch with length
\(\zeta_k\) at the position \(X_{(k)}\) on the segment
\([-\Eg,\Ed]\). Then, identify the bottom of each branch such that
\(X_{(k)}<0\) (resp. \(X_{(k)}>0\)) with the point with depth
\(\zeta_k\) on the first branch with longer length on the right
(resp. on the left). Eventually cut the semi-infinite branch at its
last (going downwards) branching point, say $\root_n$, which is at length
$\max_{1\leq 
 k\leq n} \zeta_k$. Then, consider $\root_n$ as the root of
\(\fT _n\). An instance of the ancestral tree is represented in
Fig~\ref{fig:tikz}. 
\end{enumerate}

 The next result is a consequence of 
\cite[Lemma~4.1]{abrahamExactSimulationGenealogical2021}; notice however
that in \cite{abrahamExactSimulationGenealogical2021} the ancestral
lineage (that is the position of $X_0$) is given, and that $X_0$ is not
seen as a leaf of the sampled tree. In other words, the approach
developed in \cite{abrahamExactSimulationGenealogical2021} does not involve the
immortal lineage and thus sees the
stationary CB
process $\bZ$ as a CB process with immigration, whereas our approach
here takes into account the immortal lineage as $X_0$ is a leaf of \(\fT _n\). 

\begin{lem}[Representation of the genealogical tree of $n$ individuals]
 \label{lem:sampling}
 For $n\in \N^*$, the rooted tree $\ct_n$ spanned by
 \(\cx_1, \ldots, \cx_n\) is distributed as the rooted tree
 \(\fT _n\).
\end{lem}

\usetikzlibrary{decorations}
\usetikzlibrary{decorations.pathmorphing}
\usetikzlibrary{decorations.pathreplacing}
\usetikzlibrary{decorations.shapes}
\usetikzlibrary{decorations.text}
\usetikzlibrary{decorations.markings}
\usetikzlibrary{decorations.fractals}
\usetikzlibrary{decorations.footprints}

\begin{figure}[htb]
\begin{tikzpicture}
\draw [-Stealth] (0,0) -- (11,0);
\node (0) at (0.5,0.5) {$-\Eg$};
\draw (0.5,-0.1) -- (0.5,0.1);
\node (X1) at (2,0.5) {$X_{(1)}$};
\draw[decorate,decoration={brace}]
(1.95,-1) -- (1.95,0) node[left,pos=0.5] {$\zeta_1$};
\draw (2,-0.1) -- (2,0.1);
\draw [color=red] (2,0) -- (2,-1);
\draw [dashed,color=red] (2,-1) -- (4,-1);
\node (X2) at (4,0.5) {$X_{(2)}$};
\draw[decorate,decoration={brace}]
(3.95,-2.5) -- (3.95,0) node[left,pos=0.5] {$\zeta_2$};
\draw (4,-0.1) -- (4,0.1);
\draw [color=red] (4,0) -- (4,-2.5);
\draw [dashed, color=red] (4,-2.5) -- (5.5,-2.5);
\node (X0) at (5.5,0.5) {$X_{(3)}$};
\draw (5.5,-3.5) -- (5.5,0.1);
\draw [densely dashed] (5.5,-3.5) -- (5.5,-4.7);
\draw [dashed] (5.5,-5.5) -- (5.5,-4.7); 
\node (X3) at (6.5,0.5) {$X_{(4)}$};
\draw[decorate,decoration={brace}]
(6.55,0) -- (6.55,-1.3) node[right,pos=0.5] {$\zeta_4$};
\draw (6.5,-0.1) -- (6.5,0.1);
\draw [color=red] (6.5,0) -- (6.5,-1.3);
\draw [dashed, color=red] (5.5,-1.3) -- (6.5,-1.3);
\node (X4) at (8.5,0.5) {$X_{(5)}$};
\draw[decorate,decoration={brace}]
(8.55,0) -- (8.55,-3.5) node[right,pos=0.5] {$\zeta_5$};
\draw (8.5,-0.1) -- (8.5,0.1);
\draw [color=red] (8.5,0) -- (8.5,-3.5);
\draw [dashed,color=red] (5.5,-3.5) -- (8.5,-3.5);
\node (Z0) at (10,0.5) {$\Ed$};
\draw (10,-0.1) -- (10,0.1);
\node (root) at (5,-3.5) {$\root_n$};
\node (root) at (5.5,-3.5) {$\bullet$};
\node (root) at (5,-4.7) {$\root$};
\node (root) at (5.5,-4.7) {$\bullet$};
\end{tikzpicture}

\caption{
  An instance for $n=5$  of the ancestral tree $\fT_n$ with its root
$\root_n$, which appears in Lemma~\ref{lem:sampling}. In this instance,
 the semi-infinite branch is attached to  $X_{(3)}=X_0=0$ and cut at
 the MRCA  $\root_n$ of the uniformly sampled individuals $\{X_1,
 \ldots, X_4\}$ in the whole population $[-\Eg, \Ed]$
 and $X_0$. The branch attached to $X_{(k)}$ has length $\zeta_k$, with
 $\zeta_3=0$ by convention as $X_{(3)}=0$. 
The tree $\fT'_n$ which appears in Lemma~\ref{lem:sampling2} is similar
but for  the semi-infinite branch which is now  cut at
  the MRCA $\root$ of the whole population $[-\Eg, \Ed]$. (Of course
  $\root$ is an ancestor of $\root_n$ and can be equal to $\root_n$.)}
\label{fig:tikz}
\end{figure}
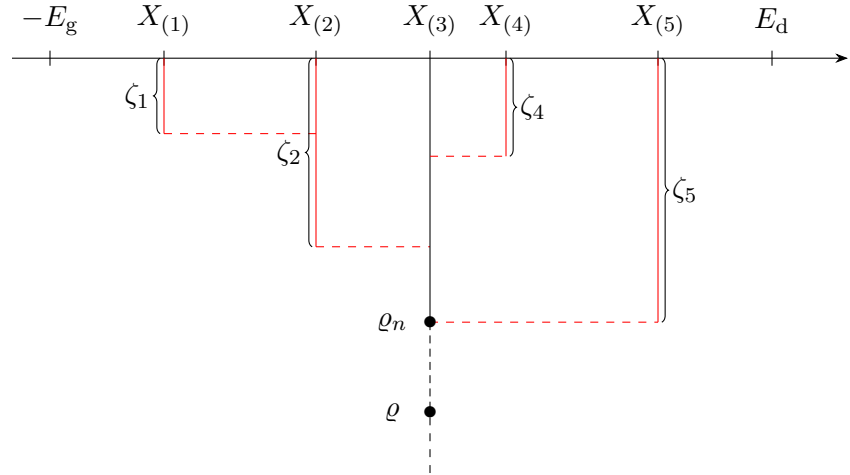

According to \cite[Proposition~7.3]{chenSmallerPopulationSize2012},
conditionally on \(Z_0\), the time $A$ of the grand MRCA of the entire
population at time \(0\) is distributed as \(\zeta^*(Z_0)\), and thus
also distributed as $\max_{0\leq k\leq n+1} \zeta_k$ conditionally on
\(\mathcal{L}_n\), that is:
\begin{equation}
  \label{eq:LoiASachantZ0}
 \P(A\le t\, |\, Z_0=z)= \P\left(\max_{0\leq k\leq n+1} \zeta_k\, |\,
   (U_k,\ k\in \N^*), \, \Eg+\Ed=z\right) = \exp(-c(t)z).  
\end{equation}
(This formula can also be deduced from~\eqref{eq:A-Nt} and the
distribution of  $N_t$.)

 We end this section with a technical lemma which will be used later
 on. Let $n\in \N^*$ be fixed. Using the ancestral process $\ca_n$
 from~\eqref{eq:def-PP-A}, we define for 
$j\leq \ell\in \llbracket 1,n\rrbracket$:
\begin{equation}
 \label{eq:def-zeta-m}
 \zm_{j:\,\ell}=\zeta_j \vee \cdots \vee \zeta_{\ell}=\sup_{j\le i\le \ell} \zeta_{i}.
\end{equation}
We also  define $ \zeta_{j:\,  \ell}^\mathrm{MRCA}$ for the time  to the
MRCA   of   $X_{(j)},\dots,X_{(\ell)}$.   By   construction,   we   have
$\zeta_{j:\,     \ell}^\mathrm{MRCA}\leq    \zm_{j:\,     \ell}$,    see
Fig.~\ref{fig:FourSituations}     for     various     instances     (and
Fig.~\ref{fig:Xj>0,blue} for an instance of strict inequality).
Notice that $\zeta_{j:\,  j}^\mathrm{MRCA}=0$ by construction and recall  that $\zeta_j=0$  if
$X_{(j)}=0$. We  have the  following  precise result.  


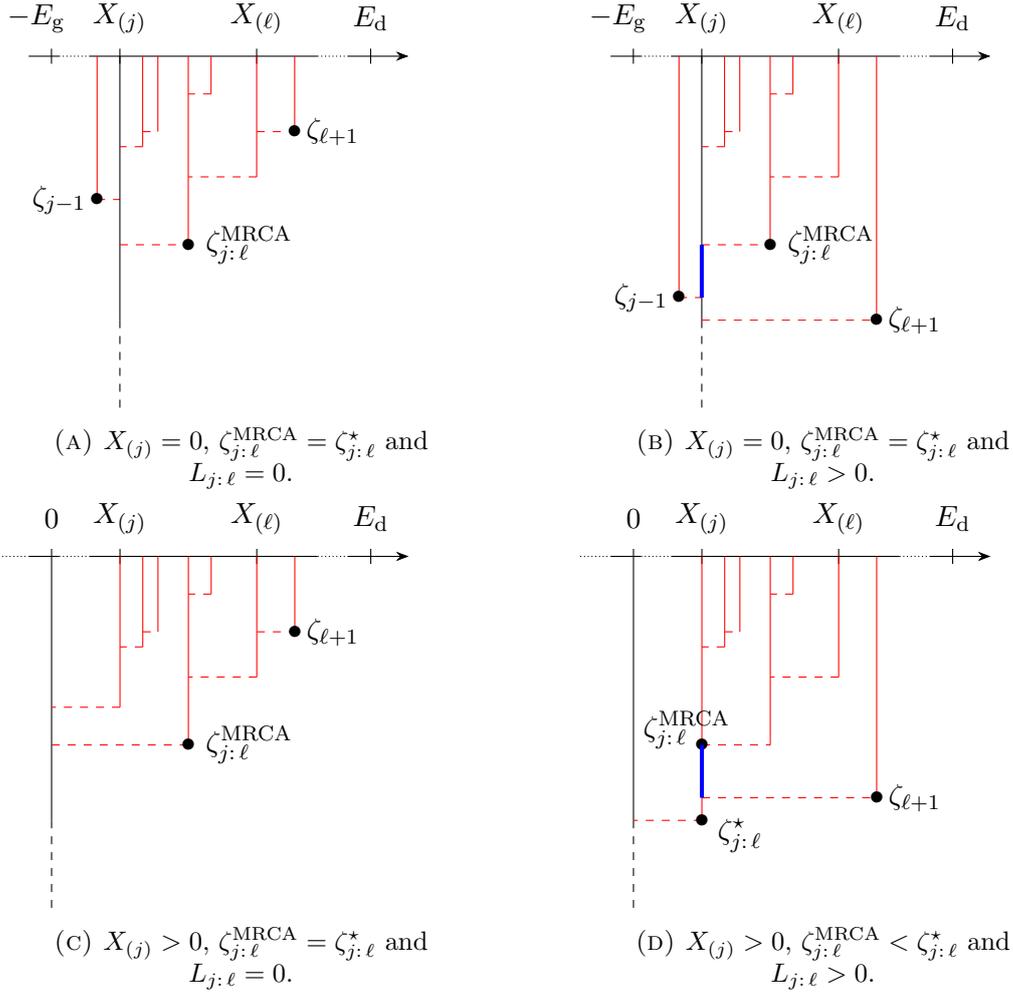
\begin{figure}[htb]
  \centering
\begin{subfigure}{0.4\textwidth}
  \begin{tikzpicture}
\draw [-Stealth] (4.2,0) -- (5,0);
\draw  (0,0) -- (0.4,0);
\draw [densely dotted] (0.4,0) -- (0.8,0);
\draw  (0.8,0) -- (3.8,0);
\draw [densely dotted] (3.8,0) -- (4.2,0);

\node (0) at (0.1,0.5) {$-\Eg$};
\draw (0.3,-0.1) -- (0.3,0.1);

\draw [color=red] (0.9,0) -- (0.9,-1.9);
\draw [dashed,color=red] (0.9,-1.9) -- (1.2,-1.9);
\node (MRCA) at (0.9,-1.9) {$\bullet$};
\node (MRCA) at (0.4,-1.9) {$\zeta_{j-1}$};

\node (X1) at (1.2,0.5) {$X_{(j)}$};
\draw (1.2,-0.1) -- (1.2,0.1);
\draw (1.2,0) -- (1.2,-3.5);
\draw [dashed] (1.2,-3.5) -- (1.2,-4.7); 

\draw [color=red] (1.5,0) -- (1.5,-1.2);
\draw [dashed,color=red] (1.5,-1.2) -- (1.2,-1.2);

\draw [color=red] (1.7,-1) -- (1.7,0);
\draw [dashed,color=red] (1.5,-1) -- (1.7,-1);

\draw [color=red] (2.1,-2.5) -- (2.1,0);
\draw [dashed,color=red] (1.2,-2.5) -- (2.1,-2.5);

\node (MRCA) at (2.1,-2.5) {$\bullet$};
\node (MRCA) at (2.9,-2.5) {$\zeta_{j:\, \ell}^\mathrm{MRCA}$};

\draw [color=red] (2.4,-0.5) -- (2.4,0);
\draw [dashed,color=red] (2.4,-0.5) -- (2.1,-0.5);

\node (X2) at (3,0.5) {$X_{(\ell)}$};
\draw (3,-0.1) -- (3,0.1);
\draw [color=red] (3,-1.6) -- (3,0);
\draw [dashed,color=red] (3,-1.6) -- (2.1,-1.6);

\draw [color=red] (3.5,-1) -- (3.5,0);
\draw [dashed,color=red] (3,-1) -- (3.5,-1);
\node (MRCA) at (3.5,-1) {$\bullet$};
\node (MRCA) at (4,-1) {$\zeta_{\ell+1}$};

\node (Z0) at (4.5,0.5) {$\Ed$};
\draw (4.5,-0.1) -- (4.5,0.1);


\end{tikzpicture}
\label{fig:Xj=0,no-blue}
\caption{$X_{(j)}=0$,  $ \zeta_{j:\, \ell}^\mathrm{MRCA}=
  \zm_{j:\,\ell}$ and $L_{j:\, \ell}=0$.}
\end{subfigure}
 \hspace{1cm}
\begin{subfigure}{0.4\textwidth}
  \begin{tikzpicture}
\draw [-Stealth] (4.2,0) -- (5,0);
\draw  (0,0) -- (0.4,0);
\draw [densely dotted] (0.4,0) -- (0.8,0);
\draw  (0.8,0) -- (3.8,0);
\draw [densely dotted] (3.8,0) -- (4.2,0);

\node (0) at (0.1,0.5) {$-\Eg$};
\draw (0.3,-0.1) -- (0.3,0.1);

\draw [color=red] (0.9,0) -- (0.9,-3.2);
\draw [dashed,color=red] (0.9,-3.2) -- (1.2,-3.2);
\node (MRCA) at (0.9,-3.2) {$\bullet$};
\node (MRCA) at (0.4,-3.2) {$\zeta_{j-1}$};

\node (X1) at (1.2,0.5) {$X_{(j)}$};
\draw (1.2,-0.1) -- (1.2,0.1);
\draw (1.2,0) -- (1.2,-3.5);
\draw [dashed] (1.2,-3.5) -- (1.2,-4.7); 

\draw [color=red] (1.5,0) -- (1.5,-1.2);
\draw [dashed,color=red] (1.5,-1.2) -- (1.2,-1.2);

\draw [color=red] (1.7,-1) -- (1.7,0);
\draw [dashed,color=red] (1.5,-1) -- (1.7,-1);

\draw [color=red] (2.1,-2.5) -- (2.1,0);
\draw [dashed,color=red] (1.2,-2.5) -- (2.1,-2.5);

\node (MRCA) at (2.1,-2.5) {$\bullet$};
\node (MRCA) at (2.9,-2.5) {$\zeta_{j:\, \ell}^\mathrm{MRCA}$};

\draw [color=red] (2.4,-0.5) -- (2.4,0);
\draw [dashed,color=red] (2.4,-0.5) -- (2.1,-0.5);

\node (X2) at (3,0.5) {$X_{(\ell)}$};
\draw (3,-0.1) -- (3,0.1);
\draw [color=red] (3,-1.6) -- (3,0);
\draw [dashed,color=red] (3,-1.6) -- (2.1,-1.6);

\draw [color=red] (3.5,-3.5) -- (3.5,0);
\draw [dashed,color=red] (1.2,-3.5) -- (3.5,-3.5);
\node (MRCA) at (3.5,-3.5) {$\bullet$};
\node (MRCA) at (4,-3.5) {$\zeta_{\ell+1}$};

\draw [line width= 1.5, blue] (1.2,-2.5) -- (1.2,-3.2);

\node (Z0) at (4.5,0.5) {$\Ed$};
\draw (4.5,-0.1) -- (4.5,0.1);


\end{tikzpicture}
\caption{$X_{(j)}=0$,  $ \zeta_{j:\, \ell}^\mathrm{MRCA}=
  \zm_{j:\,\ell}$ and $L_{j:\, \ell}>0$.}
\label{fig:Xj=0-with-blue}   
\end{subfigure}



\begin{subfigure}{0.4\textwidth}
\begin{tikzpicture}
\draw [-Stealth] (4.2,0) -- (5,0);
\draw  (0,0) -- (0.4,0);
\draw [densely dotted] (0.4,0) -- (0.8,0);
\draw [densely dotted] (-0.4,0) -- (0,0);
\draw  (0.8,0) -- (3.8,0);
\draw [densely dotted] (3.8,0) -- (4.2,0);

\node (X0) at (0.3,0.5) {$0$};
\draw (0.3,0.1) -- (0.3,-3.5);
\draw [dashed] (0.3,-3.5) -- (0.3,-4.7);


\node (X1) at (1.2,0.5) {$X_{(j)}$};
\draw (1.2,-0.1) -- (1.2,0.1);
\draw [color=red] (1.2,0) -- (1.2,-2);
\draw [dashed,color=red] (1.2,-2) -- (0.3,-2); 

\draw [color=red] (1.5,0) -- (1.5,-1.2);
\draw [dashed,color=red] (1.5,-1.2) -- (1.2,-1.2);

\draw [color=red] (1.7,-1) -- (1.7,0);
\draw [dashed,color=red] (1.5,-1) -- (1.7,-1);

\draw [color=red] (2.1,-2.5) -- (2.1,0);
\draw [dashed,color=red] (0.3,-2.5) -- (2.1,-2.5);

\node (MRCA) at (2.1,-2.5) {$\bullet$};
\node (MRCA) at (2.9,-2.5) {$\zeta_{j:\, \ell}^\mathrm{MRCA}$};

\draw [color=red] (2.4,-0.5) -- (2.4,0);
\draw [dashed,color=red] (2.4,-0.5) -- (2.1,-0.5);

\node (X2) at (3,0.5) {$X_{(\ell)}$};
\draw (3,-0.1) -- (3,0.1);
\draw [color=red] (3,-1.6) -- (3,0);
\draw [dashed,color=red] (3,-1.6) -- (2.1,-1.6);

\draw [color=red] (3.5,-1) -- (3.5,0);
\draw [dashed,color=red] (3,-1) -- (3.5,-1);
\node (MRCA) at (3.5,-1) {$\bullet$};
\node (MRCA) at (4,-1) {$\zeta_{\ell+1}$};

\node (Z0) at (4.5,0.5) {$\Ed$};
\draw (4.5,-0.1) -- (4.5,0.1);


\end{tikzpicture}
\caption{$X_{(j)}>0$,  $ \zeta_{j:\, \ell}^\mathrm{MRCA}=
  \zm_{j:\,\ell}$ and $L_{j:\, \ell}=0$.}
\label{fig:Xj>0,no-blue}   
\end{subfigure}
\hspace{1cm}
\begin{subfigure}{0.4\textwidth}
\begin{tikzpicture}
\draw [-Stealth] (4.2,0) -- (5,0);
\draw  (0,0) -- (0.4,0);
\draw [densely dotted] (0.4,0) -- (0.8,0);
\draw [densely dotted] (-0.4,0) -- (0,0);
\draw  (0.8,0) -- (3.8,0);
\draw [densely dotted] (3.8,0) -- (4.2,0);

\node (X0) at (0.3,0.5) {$0$};
\draw (0.3,0.1) -- (0.3,-3.5);
\draw [dashed] (0.3,-3.5) -- (0.3,-4.7);


\node (X1) at (1.2,0.5) {$X_{(j)}$};
\draw (1.2,-0.1) -- (1.2,0.1);
\draw [color=red] (1.2,0) -- (1.2,-3.5);
\draw [dashed,color=red] (1.2,-3.5) -- (0.3,-3.5); 

\draw [color=red] (1.5,0) -- (1.5,-1.2);
\draw [dashed,color=red] (1.5,-1.2) -- (1.2,-1.2);

\draw [color=red] (1.7,-1) -- (1.7,0);
\draw [dashed,color=red] (1.5,-1) -- (1.7,-1);

\draw [color=red] (2.1,-2.5) -- (2.1,0);
\draw [dashed,color=red] (1.2,-2.5) -- (2.1,-2.5);

\node (Z*) at (1.2,-3.5) {$\bullet$};
\node (Z*) at (1.7,-3.7) {$\zm_{j:\,\ell}$};

\node (MRCA) at (1.2,-2.5) {$\bullet$};
\node (MRCA) at (1,-2.3) {$\zeta_{j:\, \ell}^\mathrm{MRCA}$};

\draw [color=red] (2.4,-0.5) -- (2.4,0);
\draw [dashed,color=red] (2.4,-0.5) -- (2.1,-0.5);

\node (X2) at (3,0.5) {$X_{(\ell)}$};
\draw (3,-0.1) -- (3,0.1);
\draw [color=red] (3,-1.6) -- (3,0);
\draw [dashed,color=red] (3,-1.6) -- (2.1,-1.6);

\draw [color=red] (3.5,-3.2) -- (3.5,0);
\draw [dashed,color=red] (1.2,-3.2) -- (3.5,-3.2);
\node (MRCA) at (3.5,-3.2) {$\bullet$};
\node (MRCA) at (4,-3.2) {$\zeta_{\ell+1}$};

\node (Z0) at (4.5,0.5) {$\Ed$};
\draw (4.5,-0.1) -- (4.5,0.1);

\draw [line width= 1.5, blue] (1.2,-2.5) -- (1.2,-3.2);


\end{tikzpicture}
\caption{$X_{(j)}>0$,   $\zeta_{j:\, \ell}^\mathrm{MRCA}<
  \zm_{j:\,\ell}$   and $L_{j:\, \ell}>0$.}
\label{fig:Xj>0,blue}   
\end{subfigure}

\caption{Four possible configurations of \(X_{(j)},\dots,X_{(\ell)}\)
  with their TMRCA $\zeta_{j:\, \ell}^\mathrm{MRCA}$, 
along with locations (in blue and of length $L_{j:\, \ell}$) for $k$-admissible mutations (with
$k=\ell-j+1$) carried only by  
this set of leaves, whenever these exist. Notice that  $\zeta_{j:\,
  \ell}^\mathrm{MRCA}$ is strictly less than $  \zm_{j:\,\ell}$ only in
the bottom left figure.}
\label{fig:FourSituations}
\end{figure}

\begin{lemma}[Time to the MRCA of consecutive individuals]
 \label{lem:ZetaMinus}
 Let $n\in \N^*$ be given. Let $1\leq j< \ell\leq n$. We have:
\[
\zeta_{j:\, \ell}^\mathrm{MRCA} = \begin{cases}
 \zm_{j+1: \, \ell} & \text{if } X_{(j)} \geq 0, \\
 \zm_{j:\, \ell-1} &\text{if }X_{(\ell)}\leq 0, \\
 \zm_{j:\, \ell} &\text{if }
X_{(j)}X_{(\ell)} \leq 0. 
 \end{cases}
\]
\end{lemma}

\begin{proof}
 In the first case (see Fig. \ref{fig:FourSituations} on the top left 
 for an illustration), we consider that $X_{(j)}=0$, and thus 
 $\zeta_j=0$. Then, the branch with length \(\zm_{j: \, \ell}\) 
 necessarily branches on the ancestral branch of \(X_{(j)}\) (that is, 
 the branch attached 
 to $X_{(j)}$), and the branching point is the MRCA of $X_{(j)},
 \ldots,X_{(\ell)}$. Thus the time to the MRCA is
 $\zeta_{j:\, \ell}^\mathrm{MRCA}= \zm_{j+1: \, \ell}=\zm_{j: \, 
 \ell}$, where we used that $\zeta_j=0$ for the last equality.

In the second case, we consider that $X_{(j)}> 0 $ and
\(\zeta_j=\zm_{j: \, \ell}\), see an instance in
Fig.~\ref{fig:FourSituations} on the bottom right. Then, the branch with 
length \(\zm_{j+1: \, \ell}\) necessarily branches on the ancestral 
branch of \(X_{(j)}\), and the branching point is the MRCA of
$X_{(j)},\ldots,X_{(\ell)}$. This also gives
$\zeta_{j:\, \ell}^\mathrm{MRCA}= \zm_{j+1: \, \ell}$.

 In the third case, we consider that $X_{(j)}> 0 $
 and there exists $i\in \llbracket 2,\ell\rrbracket$ such that $\zeta_i
 =\zm_{j:\, \ell}$, and thus $\zeta_i
 =\zm_{j+1:\, \ell}$ (see Fig. \ref{fig:FourSituations} bottom left 
 for an illustration of this configuration). Let $i_g=\inf\{i'\in 
 \llbracket 1,j-1\rrbracket \, \colon\, \zeta_{i'}> 
 \zeta_i\quad\text{or}\quad X_{(i')}=0\}$. By definition, the 
 ancestral branch of the leaf $X_{(i)}$ branches onto the ancestral 
 branch of $X_{(i_g)}$ if $X_{(i_g)}>0$ or onto the spine if $X_{(i_g)}
 =0$. In both cases, the branching point is the MRCA of $X_{(j)},
 \ldots,X_{(\ell)}$. This also gives $\zeta_{j:\, \ell}^\mathrm{MRCA}= 
 \zeta_i=\zm_{j+1: \, \ell}$.

\medskip Those three cases give a complete picture when $X_{(j)}\geq
0$. The case $X_{(\ell)}\leq 0$ is similar. So we are left with the 
case $X_{(j)}X_{(\ell)}<0$, where one of the leaves $X_{(j)}, \ldots,
X_{(\ell)}$ belongs to the infinite spine (see Fig. 
\ref{fig:FourSituations} top right). In this case, the MRCA of 
 $X_{(j)}, \ldots, X_{(\ell)}$ is on the spine at height $\zeta_{j:\, 
 \ell}^\mathrm{MRCA}=\zm_{j: \, \ell}$,
\end{proof}

\section{Discrete Frequency Spectrum}\label{sec:DiscreteSFS}

The  neutral mutations  on the  stationary population  are given  by the
atoms of a  point measure on $\cti$ with intensity  a mutation rate, say
$\mu>0$, times the length measure  $\length(\rd x)$ on $\cti$. We sample
$n\in  \N^*$ individuals  from  the extant  population  in a  stationary
branching process  at a  given time,  say $0$  for simplicity.   In this
section, we will first give some  general results for the site frequency
spectra    of    the    ancestral    tree    \(\ct_n\),    defined    in
Section~\ref{sec:sample-T},     with    $n$     fixed.     Thanks     to
Lemma~\ref{lem:sampling},  we  can  recast  the problem  using  a  point
measure \(\cm=\sum_{i\in I_M} \delta_{m_i}\) on the random tree \(\fT _n\)
(associated  with the  ancestral point  measure $\ca_n$)  with intensity
$\mu$  times the  length measure  on its  branches. The  associated site
frequency spectrum \((\xi_k^{(n)},\ 1\le k\le n-1)\) is then defined by:
\begin{equation}\label{eq:DefSFS}
    \xi_k^{(n)} = \sum_{i\in I_M} \ind_{\{c_n(m_i)=k\}},\quad 
\end{equation}
where for \(x\in \fT _n\), \(c_n(x)\in \llbracket 1,n\rrbracket\) is the 
number of leaves $X_{(j)}$ among the \(n\) sampled leaves such that $x\preceq 
X_{(j)}$. Note that the only vertex of $\fT _n$ such that $c_n(x)=n$ 
is the root $\root$, which justifies that we are only considering the SFS 
up to index $k=n-1$.


We stress  that if a mutation is present in exactly \(k\in \llbracket
1,n-1\rrbracket
\) leaves 
of \(\fT _n\), then those leaves necessarily have consecutive positions, 
in the sense that 
the leaves carrying that mutation are exactly
$X_{(j)},\dots,X_{(j+k-1)}$
for some $j\in \llbracket 1,n-k+1\rrbracket$.
In order to be carried by exactly $k$ consecutive leaves, 
a mutation has to be ancestral to their MRCA, but no ancestral to any
other leaf. We will call such mutations 
$k$-\emph{admissible}.  

\begin{lemma}[$k$-admissible mutations]
  \label{lem:Lk}
  Let $n\in \N^*$ and $k\in  \llbracket
1,n-1\rrbracket$ be given.
 Conditionally on the ancestral point
 measure $\ca_n$, the number of 
$k$-admissible mutations carried by the \(k\)-tuple
\(X_{(j)},\dots,X_{(\ell)}\), for $j\in \llbracket 1,n-k+1\rrbracket$
and $\ell=j+k-1$, is Poisson
distributed with mean \( \mu L_{j:\, \ell}\), where: 
\begin{equation}\label{eq:Lk}
  L_{j:\, \ell} =
  \begin{cases}
    [\zeta_j \wedge \zeta_{\ell+1} - \zeta^\mathrm{MRCA}_{j:\, \ell}]_+
    & \text{if}\quad  X_{(j)}>0,\\
    [\zeta_{j-1} \wedge \zeta_{\ell} - \zeta^\mathrm{MRCA}_{j:\, \ell}]_+
    & \text{if}\quad X_{(\ell)} <0,\\
    [\zeta_{j-1} \wedge \zeta_{\ell+1} - \zeta^\mathrm{MRCA}_{j:\, \ell} ]_+
    & \text{if}\quad X_{(j)} X_{(\ell)}  \leq 0,
  \end{cases}
\end{equation}
where in~\eqref{eq:Lk}  we set
$\zeta_0=\zeta_{n+1}=+\infty $ by convention and
$\zeta^\mathrm{MRCA}_{j:\, \ell}=0$ if  $k=1$ by construction.
\end{lemma}



Intuitively, the first two cases in equation~\eqref{eq:Lk} represent the
two    symmetric   situations    in    which   all    of   the    leaves
\(X_{(j)},\ldots,X_{(j+k-1)}\) are on one side of the infinite spine. In
those  cases, $k$-admissible  mutations are  possible only  on ancestral
branches      of          \(X_{(j)},\ldots,X_{(j+k-1)}\),      see
Fig.~\ref{fig:Xj>0,blue}. The third case  represents the contribution of
the spine,  which is  nonzero if  and only  if both  \(\zeta_{j-1}\) and
\(\zeta_{j+k}\) are greater than the  longest ancestral branch among the
\(\zeta_1,\ldots,\zeta_k\) and if \(X_{(j-1)}\) and \(X_{(j+k)}\) lie on
opposite sides of the spine, see Fig.~\ref{fig:Xj=0-with-blue}.

\begin{proof} 
 We first assume that \(X_{(j)} >0\), meaning that all the $k$
 consecutive leaves $X_{(j)}, \ldots, X_{(\ell)}$ are on the right side
 of the spine. The mutations carried only by the \(k\) leaves need to
 lie on the stem of the genealogical tree, say
 $ \fT _{j:\, \ell}$, of $X_{(j)}, \ldots, X_{(\ell)}$, which
 is of length $\zm_{j:\, \ell}-\zeta_{j:\, \ell}^\mathrm{MRCA}$. By
 Lemma~\ref{lem:ZetaMinus}, this length is also equal to
 $[\zeta_j - \zm_{j+1:\, \ell}]_+ $. We shall now assume it is
 positive, that is, $\zeta_{j:\, \ell}^\mathrm{MRCA}< \zeta_j$, see
 Fig.~\ref{fig:Xj>0,blue} for an instance. 

 If
 \(\zeta_{\ell+1}\leq \zeta_{j:\, \ell}^\mathrm{MRCA}\), all mutations
 on the stem will also be carried by \(X_{(\ell+1)}\) since the
 ancestral branch of \(X_{(\ell+1)}\) will be grafted on
 $ \fT _{j:\, \ell}$, providing no $k$-admissible mutations.
 If $\zeta_{j:\, \ell}^\mathrm{MRCA}< \zeta_{\ell+1} \leq \zeta_j$,
 then the ancestral branch of $X_{(\ell+1)}$ will be grafted on the
 stem of \(\fT _{j:\, \ell}\), and only mutations between the
 root of this sub-tree and that branching point will be $k$-admissible.
 If \(\zeta_{\ell+1} > \zeta_j\), then the ancestral branch of
 \(X_{(\ell+1)} \) will be
 grafted below the stem, and all mutations on the stem are then
 $k$-admissible.

In conclusion the part of branch carrying the $k$-admissible mutations
is of length $ [\zeta_{\ell+1} \wedge \zeta_j -
\zeta^\mathrm{MRCA}_{j:\, \ell}]_+$. 

\medskip

The  case $X_{(\ell)}<0$  is  similar.   So, we  now  consider the  case
$X_{(j)}X_{(\ell)}\leq  0$,  see  Fig.~\ref{fig:Xj=0-with-blue}  for  an
instance   of   $L_{j:\,   \ell}>0$.   In   particular,   there   exists
$i\in\llbracket j,\ell\rrbracket$ (random) such that $X_{(i)}=0$, and the MRCA of
$X_{(j)},  \ldots,  X_{(\ell)}$  belongs  to the  ancestral  lineage  of
$X_{(i)}$, that is the spine.  The $k$-admissible mutations then need to
be   on   the  spine   below   the   MRCA   but   above  the   MRCA   of
$X_{(j-1)},     \ldots,     X_{(\ell)}$      and     the     MRCA     of
$X_{(j)},     \ldots,     X_{(\ell+1)}$.    Using     the     convention
$\zeta_0=\zeta_{n+1}=+\infty $,  we deduce  the part of the  branch carrying
the       $k$-admissible       mutations        is       of       length
$[\zeta_{j-1}  \wedge \zeta_{\ell+1}  - \zeta^\mathrm{MRCA}_{j:\,  \ell}
]_+$.
\end{proof}

The   number of 
$k$-admissible mutations carried by  \(\fT _n\)  is then Poisson
distributed with mean  $\mu L_k$ with:
\begin{equation}
    \label{eq:def-L}
L_k=\sum_{j=1}^{n-k+1} L_{j:\, j+k-1}.
\end{equation}
We have a simple closed formula for the expectation of $L_k$. 
Recall~\eqref{eq:law-z*}.
Let $U_{(1)} < \cdots < U_{(n)}$ be the order statistics of $n$
independent uniform random variable on $[0, 1]$ which are also
independent of $Z_0$ and of an independent exponential random variable $E$ with mean 1. 
Set $S_0=0$ and for $\ell\in \llbracket
1,n \rrbracket$:
\[
S_\ell=\E\left[ \inv{2\beta\theta}\, \log\left(1+ \frac{2\theta Z_0 \, U_{(\ell)} }{E}\right)\, \Big|\, Z_0\right]
\]

In what follows, we shall often use the digamma function  \(\Psi(x)=\Gamma'(x)/\Gamma(x)\) 
and the fact that for $x>0$:
\begin{equation}
   \label{eq:Psi}
  \log(x)  - \inv{x} \leq  \Psi(x) \leq  \log(x) -
  \inv{2x}
  \quad\text{and}\quad
  \Psi(x+1)=\Psi(x) +\inv{ x}\cdot
\end{equation}

\begin{theorem}\label{the:ExpectedSFS}
For $1\leq k\leq n-3$ we have:
\begin{multline}
    \label{eq:Espxi}
    \beta  \theta \E[\xi^{(n)}_k]= \frac{1}{2k} \left( 1 + \frac{n}{n-k-2}\right)
    + \frac{n}{(n-k)(n-k-2)} \\
    + \frac{n(n-1)}{(n-k)(n-k-1)(n-k-2)}\, \left( \Psi(k) - \Psi(n-1)\right),
\end{multline}
and for $k=n-2, n-1$:
\[
\beta  \theta   \E[\xi^{(n)}_{k_n}]=\frac{2}{3n}+ O\left(\frac{1}{n^2}\right).
\]
\end{theorem}
\begin{proof}
Let $T= E/2\theta Z_0$, where $E$ is an exponential random variable with mean 1 and independent of $Z_0$. It is elementary to check that the distribution of  $T$ has density $2(1+t)^3 \ind_{\{t>0\}}$ with respect to the Lebesgue measure. Using an elementary integration by parts, one gets for $u>0$ that:
 \[
 \E[\log(T+u)]= \log(u) - \inv{1-u} - \frac{\log(u)}{(1-u)^2}
 \quad\text{and}\quad
 \E[\log(T)]= -1. 
 \]
We recall that if $V$ has a Beta distribution with parameters $(a,b)$ such that $b>1$ then:
$\E\left[(1-V)^{-1}\right]= (a+b-1)/(b-1)$ and for $b>2$:
\[
\E[\log(V)]=\psi(a) - \psi(a+b)
\quad\text{and}\quad
\E\left[\frac{\log(V)}{(1-V)^2}\right]= \frac{(a+b-1)(a+b-2)}{(b-1)(b-2)} \Big(\psi(a ) - \psi(a+b-2)\Big),
\]
for $b=2$ and $a\in \N^*$, we have:
\[
\E\left[\frac{\log(V)}{(1-V)^2}\right]= a(a+1)  \left(\sum_{k=1}^{a-2} \inv{k^2}  - \frac{\pi^2}{6} \right) .
\]
We set:
\[
F(a,b)= (b-1) \Big(\E[\log(T+V)] - \E[\log(T)]\Big) .
\]

Recalling that $U_{(1)} < \dots < U_{(n)}$ are the order statistics of a uniform random variable, hence that $U_{(k)}$ is Beta-distributed with parameter $(n,n-k+1)$, set for $1\leq k\leq n-1$:
\[
s_k=\E[S_k]=\inv{2 \beta\theta} \frac{F(k, n-k+1)}{n-k}\cdot
\] 
By convention, we set $F(0, b)=0$. We have for $1\leq k\leq n-1$:
\[
2 \beta\theta \E[L_k]= 2 F(k, n-k+1) -  F(k+1, n-k) - F(k-1, n-k+2).
\]
Notice that if $F(a,b)$ can be written as $
\alpha + (b-1) \beta + f(a,b)$, then one gets:
\[
2 \beta\theta \E[L_k]= 2 f(k, n-k+1) -  f(k+1, n-k) -  f(k-1, n-k+2),
\]
with the convention, when $k=1$, that $f(0, b)=-\alpha- n\beta$. 
Elementary computations give, with $a+b=n+1$, $a>0$ and $b>2$:
\[
 \frac{F(a,b)}{b-1}
 =- \E[\log(T)] + \psi(a)- \psi(n+1) - \frac{n}{b-1}  - \frac{(a+ b-1)(a+b-2)}{(b-1)(b-2)} \Big( \psi(a) -  \psi(n-1)\Big),
 \]
 that is:
 \[
 F(a,b)=
-n +(b-1)\big(1 +\psi(n-1)   -  \psi(n+1)\big)   + f(a,b),
 \]
with:
\[
f(a,b)= - \frac{a(n+ b-2) }{(b-2)} \Big(\psi(a) - \psi(n-1)\Big).
\]
 Using that $\psi(k+ 1)=\psi(k)  + k^{-1}$ and, for $k\geq 2$,  $\psi(k-1)=\psi(k) - (k-1)^{-1}$, we get that for $2\leq k\leq n-3$:
 \begin{multline*}
   2 \beta\theta \E[L_k]= 
\inv{k}\left( 1 +\frac{n}{n-k-2}\right) 
 + \frac{2n}{(n-k)(n-k-2)} \\+ \frac{2n(n-1)}{(n-k)(n-k-1)(n-k-2)} \Big (\psi(k) - \psi(n-1)\Big).  
 \end{multline*}
This formula is also correct for $k=1$. This gives the result for $1\leq k\leq n-3$.

For $k=n-1$, we have:
\[
\E[\log(T+ U_{(n-1)})]
= \psi(n-1) - \psi(n+1) - n - n(n-1) \left(\sum_{k=1}^{n-2} \inv{k^2}  - \frac{\pi^2}{6} \right),
\]
and thus:
\[
2 \beta\theta s_ {n-1}= 1 + \psi(n-1) - \psi(n+1) - n + n(n-1) \sum_{k\geq n-1} \inv{k^2}\cdot
\]
Since:
\[
2 \beta \theta s_{n-2}=\frac{F(n-2, 3)}{2}=1 + \psi(n-1) - \psi(n+1) - \frac{n}{2} - \frac{(n-2)(n+1)}{2} (\psi(n-2) - \psi(n-1)),
\]
that is, $2 \beta \theta s_{n-2}=1 + \psi(n-1) - \psi(n+1) +1/2$, 
we get that:
\[
  \beta\theta \E[L_{n-1}]=  2\beta\theta  ( s_{n-1} - s_{n-2})
 =  -\inv{2}  - n + n(n-1) \sum_{k\geq n-1} \inv{k^2}
 =  \inv{2} -n + \inv{n-1} + n(n-1) \sum_{k\geq n} \inv{k^2}\cdot
\] 
Using the Euler Maclaurin expansion $ \sum_{k\geq n} \inv{k^2} =n^{-1} + (2 n^2)^{-1} + (6n^3)^{-1} + O(n^{-5})$, we get:
\[
 \beta\theta \E[L_{n-1}]=   -\inv{2}  - n + \frac{n}{n-1} + (n-1) +  \frac{n-1}{2n} + \frac{n-1}{6n^2} + O\left(\inv{n^2}\right)= \frac{2}{3n} + O\left(\inv{n^2}\right).
\]

For $k=n-2$, we have:
\begin{align*}
2 \beta \theta s_{n-3}
&=\frac{F(n-3,4)}{3}\\
&= 1 + \psi(n-1) - \psi(n+1) -\frac{n}{3} - \frac{(n-3)(n+2)}{6} \left(\psi(n-3) - \psi(n-1)\right)\\
&=1 + \psi(n-1) - \psi(n+1) -\frac{n}{3} +\frac{(n+2)(2n-5)}{6(n-2)},
\end{align*}
and thus:
\[
  \beta\theta \E[L_{n-2}]=  \beta\theta  ( 4 s_{n-2} - s_{n-1} - 3s_{n-3})
 = \frac{2}{3n} + O\left(\inv{n^2}\right).
\] 
So, for $k=n-1, n-2$, we obtain 
$\beta\theta \E[L_{n-2}]
 = \frac{2}{3n} + O\left(\inv{n^2}\right)$.
\end{proof}

\begin{lemma}[Mean of $L_k$]
  \label{lem:ELk}
 Let $n\in \N^*$ and $k\in  \llbracket
1,n-1\rrbracket$ be given. We have:
\begin{equation}
   \label{eq:ELk=}
   \E[L_k\, |\, Z_0]
=  (n-k) (2S_k-S_{k-1}- S_{k+1})
+S_{k+1}
  -  S_{k-1}   .
   \end{equation}   
\end{lemma}

\begin{proof}
    Since  \(x \wedge y = x + y
   - x\vee y \) and \([x-z]_+ = x\vee z -z\), we get that:
   \[
     [x \wedge y
     -z]_+= x \vee z+ y \vee z -z - x\vee y\vee z
     \quad\text{for all}\quad
     x,y,z\in \R.
   \]

   Set   $J\in\llbracket 1,n\rrbracket$ such that $X_{(J)}=0$.
   Using Lemmas~\ref{lem:ZetaMinus} and~\ref{lem:Lk}, we obtain for $k>1$ that:
   \begin{align*}
     \E[L_k\,& |\, Z_0]\\
     =  \E & \Big[\sum_{J<j\leq  n-k+1} \left(\zm_{j:\, j+k-1} + \zm_{j+1:\, j+k} \ind_{\{j+k\leq
             n\}} -
       \zm_{j+1:\, j+k-1} - \zm_{j:\, j+k} \ind_{\{j+k\leq
             n\}} \right) \, |\, Z_0\Big]\\
     &+ \E\Big[\sum_{1\leq j<J-k+1} \left(\zm_{j-1:\, j+k-2} \ind_{\{j\geq 2\}}+ \zm_{j:\, j+k-1} -
       \zm_{j:\, j+k-2} - \zm_{j-1:\, j+k-1}\ind_{\{j\geq 2\}} \right) \, |\, Z_0\Big]\\
    &+ \E\Big[\sum_{1 \vee (J-k+1)\leq j\leq  J \wedge (n-k+1)}
      \!\!\!\!\!\!\!\!\!\!
      \left(\zm_{j-1:\, j+k-1}\ind_{\{j\geq 2\}} + \zm_{j:\, j+k}\ind_{\{j+k\leq
             n\}} \right. \\
    & \left. \qquad\qquad\qquad\qquad\qquad\qquad - 
       \zm_{j:\, j+k-1} - \zm_{j-1:\, j+k}\ind_{\{j\geq 2, \, j+k\leq
             n\}} \right) \, |\, Z_0\Big]. 
   \end{align*}

Let $U_{(1)} < \cdots < U_{(n)}$ be the order statistics of $n$
independent uniform random variables on $[0, 1]$ which are also
independent of $Z_0$. We denote by $W_\ell$ the random variable
given by~\eqref{eq:law-z*} with $\delta$ replaced by $Z_0\, U_{(\ell)}$ and
$E$ independent of $Z_0, U_1, \ldots, U_n$. In particular, conditionally
on $J$ and $Z_0$, we have that $\zm_{j:\, \ell}$ is distributed as
$W_{\ell -j+1}$ if $J<j\leq  \ell\leq  n$ or $1\leq  j\leq  \ell <J$ but as $W_{\ell -j}$ if 
$1 \leq j\leq  J\leq  \ell \leq n$ and $j<\ell$ as $\zeta_J=0$. 
We thus deduce that:
   \begin{align*}
     \E[L_k\, |\, Z_0]
     =  \E & \Big[\sum_{J<j\leq  n-k+1} \left(W_k + W_k\ind_{\{j+k\leq
             n\}}  -
 W_{k-1}  - W_{k+1}\ind_{\{j+k\leq
             n\}}  \right) \, |\, Z_0\Big]\\
     &+ \E\Big[\sum_{1\leq j<J-k+1} \left( W_k +W_k\ind_{\{j\geq 2\}} -
       W_{k-1} -W_{k+1}\ind_{\{j\geq 2\}} \right) \, |\, Z_0\Big]\\
    &+ \E\Big[\sum_{1 \vee (J-k+1)\leq j\leq  J \wedge (n-k+1)}
      \!\!\!\!\!\!\!\!\!\!
      \left(W_{k}\ind_{\{j\geq 2\}} + W_{k}\ind_{\{j+k\leq
             n\}} \right. \\
    & \left. \qquad\qquad\qquad\qquad\qquad\qquad  - W_{k-1}  - W_{k+1}\ind_{\{j\geq 2,\,  j+k\leq
             n\}} \right) \, |\, Z_0\Big].
   \end{align*}
   By definition, we have
\(S_\ell=\E[W_\ell\, |\, Z_0]\) for $\ell\in \llbracket
1,n \rrbracket$. We get:
   \begin{align*}
   \E[L_k\, |\, Z_0]
&  =  2(n-k) S_k
  - (n-k+1) S_{k-1}
  - (n-k-1) S_{k+1}\\
&  =  (n-k) (2S_k-S_{k-1}- S_{k+1})
+S_{k+1}
  -  S_{k-1}   .
  \end{align*} 
It is easy to check that this formula also holds for $k=1$ as $S_0=0$. 
\end{proof}

We now compute the SFS  of  the  ancestral  tree
\(\ct_n\) of  $n\in \N^*$ individuals  sampled from  the extant population in a stationary
branching process at  a given time, say $0$ for  simplicity.   

\begin{theorem}[Site  frequency spectra of $n$ individuals at a given generation]
  \label{the:SFS_discret}
 The expected number of mutations
carried by exactly \(k\in \llbracket 1, n-1 \rrbracket \) individuals
among $n\geq 2$ individuals sampled uniformly in the
population at a fixed time for a stationary subcritical branching
process satisfies:
\[
   \frac{\beta}{\mu Z_0}\,    \E[\xi_k^{(n)}  |\, Z_0] = \inv{k} +
   \inv{k} g_1\left(\theta Z_0, 
      \frac{k}{n}\right)+  \frac{\sqrt{k} }{n^2}\, g_2\left(\theta Z_0,
      \frac{k}{n}, n\right),
  \]
where the function $g_1$ is continuous on $\R_+^*\times [0,1]$ with $g_1(z, 0)=0$ and for all $z$, there exists a constant $C$ such that  $ g_1(z,u)\leq C u (|\log(u)| +1)$ 
and $ g_2(z,u, n)\leq C$ for all $u\in [0, 1]$ and $n\geq 2$. 
In particular, if $(k_n, n\in \N^*)$ is a sequence such that $\lim_{n\rightarrow
  \infty} k_n/n=u\in [0,1]$ and \(k_n\in \llbracket 1, n-1 \rrbracket \), then we have:
\begin{equation}\label{eq:EsperanceSFS}
\lim_{n\rightarrow \infty}  k_n \, \E[\xi_{k_n}^{(n)}  |\, Z_0] =
\frac{\mu Z_0}{\beta}\left(1+g_1\left(\theta Z_0,
      u\right)\right).
\end{equation}
\end{theorem}
The function $g_1$ is explicitly given in \eqref{def:g1} and drawn in Fig. \ref{fig:g1_plot} for various values of $z$. Note that $g_1(z,u)=2(z-1)u \log(u) + O(u)$ for small $u$, hence $g_1$ is not differentiable at $u=0$ except for the singular value $z=1$, which corresponds to the case in which $Z_0$ is equal to its mean $\E[Z_0]=1/\theta$. In particular, for $\lim_{n\rightarrow
  \infty} k_n/n=u\in [0,1]$ and $u$
small, we get:
\begin{equation}
    \label{eq:DL-gn}
 \lim_{n\rightarrow \infty}  k_n \, \E[\xi_{k_n}^{(n)}  |\, Z_0] \simeq 
\frac{\mu Z_0}{\beta}
\big(1+ 2(\theta Z_0-1) u\log(u)\big).
\end{equation}

\begin{figure}
    \centering
    \includegraphics[width=.7\textwidth]{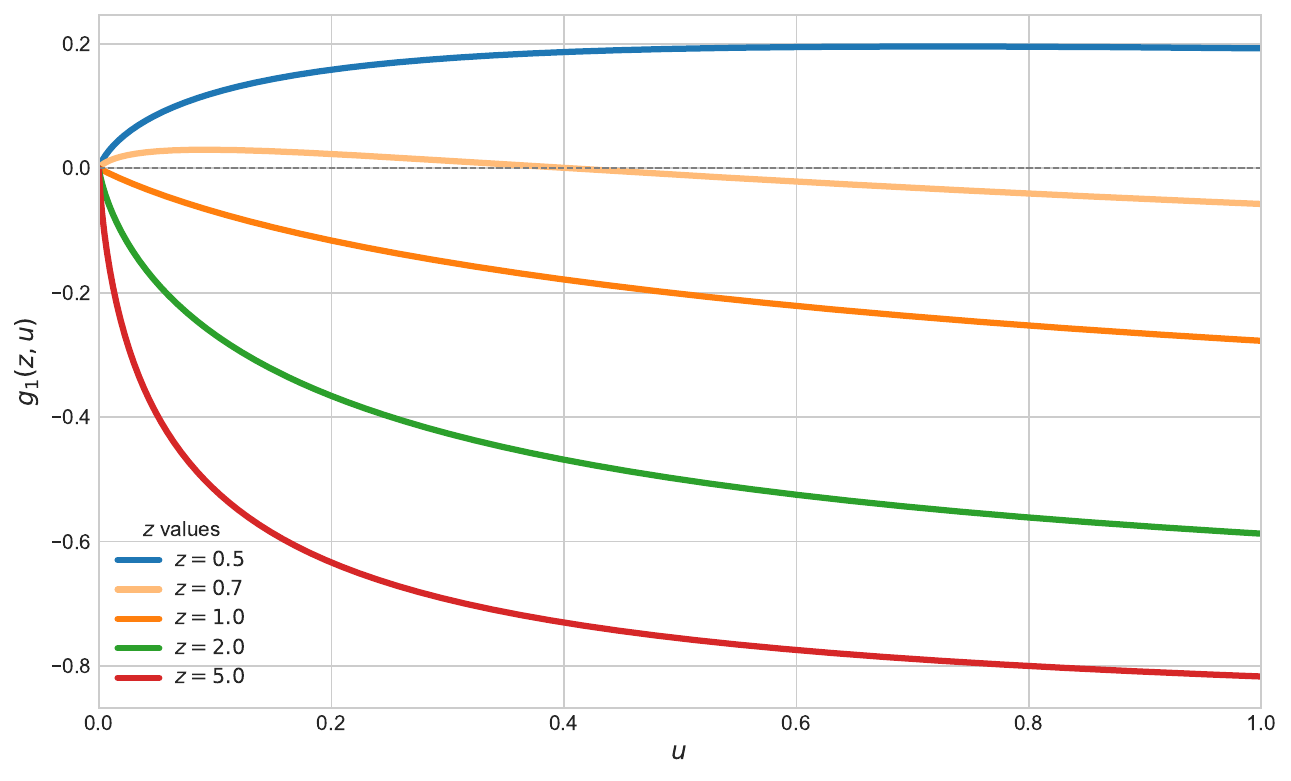}
    \caption{Plot of $g_1(z,u)$ for various values of $z$.}
    \label{fig:g1_plot}
\end{figure}

\begin{proof}
  Thanks  to Lemma~\ref{lem:sampling},  we recast  the problem  using a
   point measure  \(\mathcal{M}\) on the random  tree \(\fT _n\)
   (associated with the ancestral  point measure $\ca_n$) with intensity
   $\mu$   times    the   length   measure   on    its   branches.   Let
   $J\in  \llbracket  1,n\rrbracket $  such  that  $X_{(J)}=0$. We  also
   recall the  definition~\eqref{eq:def-zeta-m} of  $\zm_{j:\,\ell}$ and
   that $\zeta_{J}=0$. 


   Let   $k\in    \llbracket   1,n-1\rrbracket$.    The  number   of
   $k$-admissible mutations carried by \(\fT _n\) is Poisson distributed
   with mean $\mu L_k$ given  in~\eqref{eq:def-L}.  Recall that $S_k$ is
   distributed as  $\E[\zeta^*(Z_0 U_{(k)})|Z_0]$ with $U_{(k)}$  the $k$-th
   order    statistic    of    $n$   independent    random    variables
   $U_1, \ldots,U_n$ uniformly distributed over $[0,1]$ also independent
   of   $Z_0$,   and   $(Z_0,    U_{(k)})$   is   independent   of   $E$
   in~\eqref{eq:law-z*}.     We     also    recall     formula    (5.15)
   from~\cite{abrahamExactSimulationGenealogical2021}:
\[
  \E\left[\zeta^*(\delta) \right]= (\delta/\beta) \, H(2\theta \delta)
  \quad\text{with}\quad
  H(x)=\int_0^\infty  \frac{1- \expp{-u}}{u}\, \frac{\rd u}{u+x}\cdot
\]



Let  $\gamma$ be the Euler constant. Using that:
\[
  1-\gamma= \int_0^1 \left(1- \expp{-u} -u\right)\, \frac{\rd u}{u^2} +
  \int_1^\infty  \left(1- \expp{-u} \right)\, \frac{\rd u}{u^2}, 
\]
and elementary computations, we get that:
\begin{align*}
  H(x)
&  = \int_0^1 \left(1- \expp{-u}  - u + \frac{u^2}{2} -
     \frac{u^3}{6}\right) \left( \inv{u+x} -\inv{u}\right)\, \frac{\rd
     u}{u}
     + \int_1^\infty  (1- \expp{-u})   \left( \inv{u+x} -\inv{u}\right)\, \frac{\rd
     u}{u}\\
 &\hspace{2cm}  + \int_0^1 \left(1 - \frac{u}{2} + \frac{u^2}{6}\right) \frac{\rd
    u}{u+x}
    + \int _0^1 \left(\inv{2} - \frac{u}{6}\right) \, \rd u + 1-\gamma\\
&=  h_0(x)- \left(1+ \frac{x}{2}+ \frac{x^2}{6}\right) \log(x)-
  \frac{x}{6} + 1-\gamma ,
\end{align*}
where:
\begin{equation}
   \label{eq:def-h0}
  h_0(x)= \left(1 + \frac{x}{2} + \frac{x^2}{6}\right)\log(1+x) -
  \int_{(0, \infty )} f(u) \,  \frac{x\, \rd u}{u+x},
\end{equation}
and:
\begin{equation}
   \label{eq:def-f}
  f(u)=\inv{u^2} \left(1- \expp{-u} -u+
    \frac{u^2}{2}-  \frac{u^3}{6}\right) \ind_{\{u\leq 1\}} 
  +\inv{u^2} \left(1- \expp{-u}\right) 
  \ind_{\{u>1\}}.
\end{equation}
This decomposition is motivated by the fact that $ f(u)/u^2$ is
integrable. So we get:
\[
 \beta  \E\left[\zeta^*(\delta) \right]= - \delta \left( 1+\theta
 \delta + \frac{2}{3} \theta^2 \delta^2\right)  
\log(2\theta\delta)
+ (1- \gamma) \delta
  - \frac{ \theta}{3} \delta^2
  + \delta  h_0(2\theta\delta).
\]
Let $\log_+(x)=\max(0, \log(x))$. For simplicity, we set:
\begin{equation}
   \label{eq:def-h1}
  h_1(x)=xh_0(x),
\end{equation}
and get  the following bounds on the
derivatives of $h_1$: there exists a
finite constant $C$ such that for $x\geq  0$:
\begin{equation}
   \label{eq:majo-h0}
  |  h_1^{(i)}(x)|\leq   C(1 +x^{3-i}\log_+(x))
  \quad\text{and}\quad
  i\in\{0, \ldots, 3\}.
\end{equation}


 Now, in the computation of $S_k$, the random variable $U_{(k)}$,  which is independent of $Z_0$, has
 a   Beta    distribution   with   parameter   $(k,n-k+1)$. We recall
 that  if $V$ has
 a   Beta    distribution   with   parameter   $(a,b)$ and $i\in \N$:
  \begin{align*}
    \E[V]&=\frac{a}{a+b},\\
    \E[V^2]
               &=\frac{a(a+1)}{(a+b)(a+b+1)},\\
    \E[V^{i+1} \log(V) ]
               &= \frac{a\cdots (a+i)}{(a+b)\cdots
                 (a+b+i)}\big(\Psi(a+i+1)-\Psi(a+b+i+1)\big) .
  \end{align*}
The formula for the geometric mean $\Es[\log(V)]$ is well-known (see e.g. (25.17)' in \cite{johnson1995continuous}) and the general formula for $\Es[V^{i+1}\log(V)]$ can be derived using the fact that $\Es[V^{i+1}\log(V)]=\Es[\log(V')]\mathrm{B}(a+i+1,b)/\mathrm{B}(a,b)$, where $V'$ is $\mathrm{Beta}(a+i+1,b)$-distributed. Let us mention also that $\Psi(\ell+1)=H_{\ell} - \gamma$, with 
$H_0=0 $ and \(H_\ell=\sum_{i=1}^\ell i^{-1}\) the harmonic sum for $\ell\in \N^*$
and $\gamma$ the Euler constant.

We set:
\begin{align*}
  B_{n,k}&=\Psi(n+2) -\Psi(k+1)  +  1-\gamma -  \log(2\theta Z_0),\\
  C_{n,k}&=\Psi(n+3) -\Psi(k+2)   - \inv{3} - \log(2\theta Z_0),\\
  D_{n,k}&= \Psi(n+4) -\Psi(k+3)   -  \log(2\theta Z_0)
\end{align*}
and 
\[
     A_{n,k}=\frac{k}{(n+1)} B_{n,k} + \theta Z_0\, \frac{
       k(k+1)}{(n+1)(n+2)} C_{n,k}+ \frac{2\theta^2 Z_0^2}{3} 
     \frac{k(k+1)(k+2)}{(n+1)(n+2)(n+3)} 
     D_{n,k},
\]
as well as:
\[
  F_{n,k}=\E\left[U_{(k)}\,  h_0(2 \theta Z_0\, U_{(k)})\,|\,
    Z_0\right]=\inv{2 \theta Z_0} \E\left[ h_1(2 \theta Z_0\, U_{(k)})\,|\, Z_0\right]. 
\]
In particular, we have:
\begin{equation}
   \label{eq:S=A+F}
\frac{\beta}{Z_0}\,    S_k=   A_{n,k}+
  \,F_{n,k}.
\end{equation}
By convention we set $U_{(0)}=0$ so that the formula~\eqref{eq:S=A+F}
also holds for $k=0$ as by convention  $S_0=0$. 
Recall $k\in    \llbracket   1,n-1\rrbracket$. Using~\eqref{eq:Psi}, we also get:
 \begin{align*}
\frac{\beta}{Z_0}\,    S_{k-1} &=    A_{n,k}+ F_{n,k-1} -
             \frac{1}{(n+1)}\left(B_{n,k} -1 
             +\inv{k}\right)  \\
   & \hspace{2cm}-  \theta Z_0
             \frac{k}{(n+1)(n+2)} \left(2 C_{n,k} - \frac{k-1}{k+1}\right)\\
      & \hspace{2cm} - \frac{2 \theta^2 Z_0^2}{3}\,
        \frac{k(k+1)}{(n+1)(n+2)(n+3)} \left( 3D_{n,k} - 
        \frac{k-1}{k+2}\right), 
   \\
   \frac{\beta}{Z_0}\,     S_{k+1}
                               &=   A_{n,k} + F_{n,k+1} 
  + \frac{1}{(n+1)}(B_{n,k}
   -1) \\
      & \hspace{2cm} +  \theta Z_0
        \frac{k+1}{(n+1)(n+2)} (2 C_ {n,k}- 1)     \\
   & \hspace{2cm} + \frac{2 \theta^2 Z_0^2}{3}\,
     \frac{(k+1)(k+2)}{(n+1)(n+2)(n+3)} \left( 3D_{n,k}-1 \right). 
 \end{align*}
So, using Lemma \ref{lem:ELk}, for $k\in  \llbracket 1,n-1\rrbracket$, we have:
\begin{multline}
  \label{eq:synthese}
    \frac{\beta}{Z_0}\,   \E[L_k\, |\, Z_0] =    \inv{k}+  
  \frac{B_{n,k}^{(1)}}{(n+1)}
     +  \theta Z_0\,  \frac{C_{n,k}^{(1)}}{(n+1)(n+2)}
\\
+\frac{2 \theta^2 Z_0^2}{3}\frac{(k+1)D_{n,k}^{(1)} }{(n+1)(n+2)(n+3)}
+    \left(R^{(1)}_{n,k}+  (n-k)R^{(2)}_{n,k} \right).
 \end{multline} 
 with:
 \begin{align*}
   B_{n,k}^{(1)}&=  2B_{n,k} -3,\\
   C_{n,k}^{(1)}&= 2(-n+3k+1) C_{n,k} +
                  \frac{n(3k+1) - (5k ^2 +2k +1)}{k+1} ,\\
D_{n,k}^{(1)} &=  6 (-n+2k+1)
  D_{n,k} + \frac{1}{k+2} ((n-k)(5k+4) - (2k^2+ 3k +4)),\\
  R^{(1)}_{n,k}&= F_{n,k+1} -F_{n,k-1},\\
   R^{(2)}_{n,k}&= \big(2F_{n,k} - F_{n,k-1} -F_{n,k+1}\big).
 \end{align*}
So we get with $u=k/n\in (0,1)$:
\begin{align*}
  \frac{B_{n,k}^{(1)}}{(n+1)}
  &= \frac{u}{k} \left(
  -2\log(u) -1 -2\gamma - 2\log(2\theta 
      Z_0)\right) + O\left(\frac{u}{k^2} \right),\\
  \frac{C_{n,k}^{(1)}}{(n+1)(n+2)}
  &=\frac{u}{k} \left(2(1-3u) \left(\log(u) 
+ \log(2\theta Z_0)\right) +\frac{11}{3} - 7 u
    \right) + O\left(\frac{u}{k^2} \right),\\
  \frac{(k+1) D_{n,k}^{(1)}}{(n+1)(n+2)(n+3)}
  &=\frac{u^2}{k} \left(6(1-2u )
    \left(\log(u) + \log(2\theta Z_0)\right) + 5 
    - 7u 
    \right) +  O\left(\frac{u}{k^2} \right),
\end{align*}
where $O\left(u/k^2 \right)$  has to be understood  as a function
of $\theta$, $Z_0$, $n$, and $k$ which  is bounded by $C/nk$, with $C$ a
constant depending only on $\theta$ and $Z_0$. We first consider the term $R^{(1)}_{n,k}$:
\[
2\theta Z_0\,   R_{n,k}^{(1)}=\E\left[h_1(2 \theta Z_0\, (\Delta+ U_{(k-1)})) - h_1(2
    \theta Z_0\, ( U_{(k-1)}))) \, |\, Z_0
  \right],
\]
where $\Delta=U_{(k+1)} -U_{(k-1)}$ is distributed as $U_{(2)}$.
Recall $u=k/n$, and notice that:
\begin{equation}
  \label{eq:EU-u}
   \E\left[(U_{(k-1)}
-u) ^2\right] = \frac{u(1-u)}{n} + O(n^{-2}). 
\end{equation}
Notice that~\eqref{eq:EU-u} holds indeed for $k=1$ as by convention
$U_{(0)}=0$ and the left hand-side of~\eqref{eq:EU-u} is equal to $O(1/n^2)$. 
Since:
 \[
h_1(\delta+x) -h_1(x)= \delta h_1'(x)+ \int_0^\delta (\delta-t)
h''_1(t+x)\, \rd t,
\]
and, thanks to~\eqref{eq:majo-h0} for the control of the second
derivative of $h_1$:
\[
  |h_1'(2 \theta Z_0 \, U_{(k-1)}) - h_1'(2 \theta Z_0 u)|
  \leq  C  (1+\theta Z_0)^3  |U_{(k-1)} -u|
\]
we deduce, using the Cauchy-Schwarz  inequality
and~\eqref{eq:EU-u},  that:
\begin{align} \label{eq:R_nk1}
  R_{n,k}^{(1)}
  &= 
     \E[ \Delta]h'_1(2 \theta Z_0\, u) +  \theta Z_0\, \E[
    \Delta^2]^{1/2} \,  \E\left[(U_{(k-1)}
-u) ^2\right]^{1/2} \, 
    O\left(1\right) 
  + O\left(\E[\Delta^2] \right)\\
&  = \frac{2}{k} \,u h'_1\left(2 \theta Z_0\,u\right) +
  O\left(\frac{u^2}{k^{3/2}}\right).\nonumber
\end{align}
As above, in this equation, the $O(r)$ term should be understood as bounded by $Cr$, where the constant $C$ 
depends only on $\theta$ and $Z_0$. We now control the term $ R_{n,k}^{(2)}$. We have:
\begin{equation}
   \label{eq:DL-h0}
  2 h_1(x+ \delta) - h_1(x) - h_1(x+\delta+\delta')= (\delta-\delta')
h_1'(x) + H(x, \delta, \delta'),
\end{equation}
with:
\begin{multline}
   \label{eq:def-Hdd}
   H(x, \delta, \delta')\\
   \begin{aligned}
&=  2 \int_0^{\delta} (\delta-t) \, h_1''(x+t)\,  \rd t\ -
 \int_0^{\delta+\delta'} (\delta +\delta'-t) \, h_1''(t)\,  \rd t\\
 &    = \left(\delta^2 - \frac{(\delta+
    \delta')^2}{2}\right) h_1''(x) + \int_0 ^\delta\!\!\! (\delta -t)^2 h_1'''
   (t+x)\, \rd t -  \inv{2} \int_0 ^{\delta+ \delta'} \!\!\!\!\!\!\!\! (\delta -t)^2 h_1'''
  (t+x)\, \rd t .
   \end{aligned}
\end{multline}
Take $X=2\theta Z_0 U_{(k-1)}$, $\delta=2\theta Z_0 \Delta$ and
$\delta'=2\theta Z_0 \Delta'$ with $\Delta=U_{(k)} -U_{(k-1)}$ and 
$\Delta'=U_{(k+1)} -U_{(k)}$. Notice that $\Delta$ and $\Delta'$ are
distributed as $U_{(1)}$ and that $(\Delta, U_{(k-1)})$ and $(\Delta',
U_{(k-1)})$ have the same distribution. This implies that:
\[
  \E[ (\Delta -\Delta') h'_1( X)\, |\,  Z_0]=0,
\]
and thus:
\begin{equation}
  \label{eq:R2-v2}
  2\theta Z_0\,     R_{n,k}^{(2)}
= \E[2h_1(X+\delta) - h_1(X) -h_1(X+\delta+\delta')]
  = \E[H(X, \delta, \delta')\,|\, Z_0] .
  \end{equation}
We also have, thanks to~\eqref{eq:majo-h0}:
  \[
  |h_1''(2 \theta Z_0 \, U_{(k-1)}) - h_1''(2 \theta Z_0 u)|
  \leq  C  (1+\theta Z_0)^2  |U_{(k-1)} -u|.
\]

Using~\eqref{eq:R2-v2}, \eqref{eq:EU-u} and that $\Delta+\Delta'$ is distributed as $U_{(2)}$, we obtain similarly that:
\begin{align*}
  R_{n,k}^{(2)}
  &= 2 \theta Z_0  \left(\E[\Delta^2] - \frac{\E[(\Delta+\Delta')^2]}{2}  \right) 
    h''_1(2 \theta Z_0\, u)
    +  \inv{n} O\left(\frac{u^2}{k^{3/2}}\right)  \\
 &= - 2 \theta Z_0 \frac{ u }{nk} \, 
    h''_1\left(2 \theta Z_0\, u\right)
    +  \inv{n} O\left(\frac{u^2}{k^{3/2}}\right) .
\end{align*}
We thus obtain that:
\begin{equation}
   \label{eq:R-final}
 R^{(1)}_{n,k}+  (n-k)R^{(2)}_{n,k} =   \frac{2}{k}
 \,u h'_1\left(2 \theta Z_0\,u\right) -
  \frac{2\theta Z_0}{k}  u(1-u)  \, 
    h''_1\left(2 \theta Z_0\, u\right)
 + O\left(\frac{u^2}{k^{3/2}}\right).
\end{equation}

\medskip

In conclusion, we get that for $k\in  \llbracket
1,n-1\rrbracket$:
\begin{equation}
   \label{eq:compil}
  \frac{\beta k}{Z_0}\,   \E[L_k\, |\, Z_0]
= 1 + g_1(\theta Z_0, u)
+ O\left(\frac{u^2}{k^{1/2}}\right)
= 1 + g_1(\theta Z_0, u)
 + \frac{k^{3/2}}{n^2}\, O\left(1\right),
\end{equation}
with $g_1$ given for $u\in [0, 1]$ and $z>0$ by:
\begin{align}
  \label{def:g1}
  g_1(z, u)
  =&
u  \left(
  -2\log(u) -1 -2\gamma - 2\log(2z)\right) \\
  \nonumber
   &+    z u
 \left(2(1-3u) \left(\log(u) 
+ \log(2z)\right) +\frac{11}{3} - 7 u
\right)\\
  \nonumber
 &+\frac{2}{3} \,z^2 u^2 \left(6(1-2u )
    \left(\log(u) + \log(2z)\right) + 5 
    - 7u 
  \right)\\
  \nonumber
   &+    2
 u \, h'_1\left(2z u\right) -
2 z u(1-u)  \, 
    h''_1\left(2z u\right),
\end{align}
where $h_1$ is defined in~\eqref{eq:def-h1} through $h_0$
from~\eqref{eq:def-h0} and $f$ from~\eqref{eq:def-f}. Thanks to~\eqref{eq:majo-h0}, we get that $g$ is continuous on
$\R_+^* \times [0,1]$, that 
$g_1(z,0)=0$ and  that for all $z>0$,
there exists a constant $C$ such 
that $g_1(z, u) \leq C u(|\log(u)|+1)$.
Set $g_2(z,u,n)$ as $n^2/k^{3/2}$ times the very last right-hand side term of~\eqref{eq:compil}
so that $  g_2(z,u,n)=O(1)$. Then, use  Lemma~\ref{lem:Lk}, to get  $ \E[\xi_k^{(n)}  |\, Z_0] = \mu \E[L_k\, |\, Z_0] $.

Note that for $k=n-1$, formula \eqref{eq:Lk} reduces to $\E[L_{n-1}|Z_0] = 2(S_{n-1}-S_{n-2})$. The derivations above are still valid in that case, except for $R_{n,n-1}^{(1)}=2(F_{n,n-1}-F_{n,n-2})$ and $R_{n,n-1}^{(2)}=0$. The same computations (with $\Delta=U_{(n-1)}-U_{(n-2)}$ distributed as $U_{(1)}$) give the same asymptotic \eqref{eq:R_nk1} for $R_{n,n-1}^{(1)}$. Since for $u=1$, the second derivative term $u(1-u)h_1''(2\theta Z_0 u)$ vanishes in \eqref{eq:R-final}, this enables us to recover the asymptotic \eqref{eq:EsperanceSFS}.
\end{proof}

\section{Continuous Frequency Spectrum}\label{sec:ContinuousSFS}

In this section,  we will consider the continuous  frequency spectrum of
the genealogical  tree $\cti$  associated to  the stationary  CB process
$\bZ$. We consider a (neutral) mutation process given by 
 a Poisson point process on $\cm=\sum_{i\in I_M} \delta_{m_i} $ on \(\cti\) with 
intensity \(\mu\length(\rd x)\), where \(\mu>0\) is the individual mutation
rate and $\length (\rd x)$ the length measure on $\cti$. 

The total offspring subtree of
 \(x\in \cti\) is defined by $\cti(x)=\{y\in \cti\, \colon\, x \preceq
 y\}$ and the corresponding clonal sub-tree is defined by:
\begin{equation}
   \label{eq:clonal-tree}
    \ctcl(x)=\{y\in \cti(x)\, \colon\, \cm(\llbracket x,y
   \rrbracket)=0\}.
\end{equation} 
In the following sections we shall study the mean measure of 
the size of the 
population at time 0 carrying at least one mutation and the size of the
clonal
population at time $0$ of the MRCA of  the  extant
population at time 0.

\subsection{The mean site frequency measure}
 
Recall that for $t\in \R$, $\cz_t(\cti)$ is the local time at level $t$ of the infinite tree representing the 
genealogy of the stationary CB process $\bZ$. In particular, at a given level $t$, $\cz_t(\cti)$ is distributed 
as $Z_t$, which in turn is distributed as the sum of two independent exponential random variables with 
parameter $2\theta$. Following 
\cite{Duchamps2018}, we consider the  site frequency point measure on
$(0, +\infty )$  of the
 extant population at time $0$:
\begin{equation}\label{eq:MesuresFrequence}
  \Pt = \sum_{i\in I_M} \delta_{\cz_0(\cti(m_i))}.
\end{equation}
In other words, we associate to each mutation on the tree the size of the 
population at time 0 carrying it.
The main result of this section describes the mean measures \(\Lambda\) 
 of this point measure:
\begin{equation}
 \Lambda(\rd r) = \Es[\Pt(\rd r)].
\end{equation}
Let  \(\Gamma(0,r)=\int _r^\infty  v^{-1}\expp{-v}\, \rd v\) denote the incomplete Gamma function.
\begin{theorem}[The mean SFS measure] \label{the:meanSFS}
  The mean site frequency measure \(\Lambda\) of the genealogical tree
  $\cti $ (associated to the  stationary CB process $\bZ$) is absolutely
  continuous with respect to the Lebesgue measure on \(\mathbb{R}_+\), with
  density given by:
 \begin{equation}\label{eq:DensityLambdat}
 f(r) = \frac{\mu}{\beta} \left(\frac{\expp{-2\theta r}}{\theta r} + 
 \expp{-2\theta r}+2\theta r \, \Gamma(0,2\theta r)\right). 
 \end{equation}
\end{theorem}
In this formula, as will be apparent from the proof, the first term represents the contribution from the mutations not located on the spine, whereas the second and third term come from the spine mutations. It is worth noticing that
\[
  f(r)\sim_{r\to 0+} \frac{\mu}{\beta \theta r}
  \quad\text{and}\quad
f(r)\sim_{r\to +\infty} \frac{2\mu}{\beta} 
\expp{-2\theta r}.
\]
In other words, for small \(r\), that is, for mutations shared
by a small fraction of the extant population at time 0, the only significant 
contribution comes from non-spine mutations. By contrast, 
for large \(r\), corresponding to mutations shared by a large number of the 
extant population, only spine mutations are significant. 

\begin{rem}[Discrete SFS measure]
    \label{rem:discrete-sfs-measure}
The discrete site frequency point measure is the following measure, defined on $(0,1]$, with $\xi_{k}^{(n)}$ as in \eqref{eq:DefSFS}:
\[
    \Phi_\mathrm{d}^{(n)} = \sum_{1\le k\le n-1} \xi_{k}^{(n)} \delta_{\frac{k}{n}}.
\]
Note that in this case, we are normalizing space, thus considering the fraction $k/n$ of the population carrying a given mutation rather than the number of individuals. As $n$ goes to $\infty$, conditionally on $\cti$, this measure converges (when tested against continuous function with compact support in $(0, 1]$)
a.s.\ to the normalized site frequency point measure:
\[
    \Phi_\mathrm{d}^{(\infty)} = \sum_{i\in I_M} \delta_{\cz_0(\cti(m_i))/\cz_0(\cti)}.
\]
Let $f$ be a nonnegative continuous function with support in $[a,b]\subset (0, 1]$. 
We have:
\begin{align*}
\E[\langle \Phi_\mathrm{d}^{(n)}, f \rangle \,|\, Z_0]
&= \inv{n} \sum_{k\in [an, bn]} n \E[ \xi_k^{(n)}\,|\, Z_0]\,  f\left( \frac{k}{n}\right)\\
&= \frac{\mu Z_0}{\beta}\, \inv{n} \sum_{k\in [an, bn]} \left(1+ g_1\left( \theta Z_0, \frac{k}{n}\right) \right)  \frac{n}{k} \ f\left( \frac{k}{n}\right) + O( n^{-1/2}).
\end{align*}
Using the continuity of $g_1$, we deduce that a.s.:
\[
\lim_{n \rightarrow \infty} \E[\langle \Phi_\mathrm{d}^{(n)}, f \rangle \,|\, Z_0]= \frac{\mu Z_0}{\beta}\,\int_0^1 \frac{1+ g_1\left( \theta Z_0, u\right)}{u} \, f(u)\, \rd u .
\]
Similarly, we have using~\eqref{eq:def-bf:g}:
\[
\lim_{n \rightarrow \infty} \E[\langle \Phi_\mathrm{d}^{(n)}, f \rangle ]
= \int_0^1 \frac{\mathbf{g}(u)}{u} \, f(u)\, \rd u .
\]
It is thus natural to conjecture that:
\[
\E[ \Phi_\mathrm{d}^{(\infty)} (\rd r) \,|\, Z_0]=\frac{\mu Z_0}{\beta}\,  \ind_{(0, 1]}(r) \, \frac{1+ g_1(\theta Z_0, r)}{r} \, \rd r 
\quad\text{and}\quad
\E[ \Phi_\mathrm{d}^{(\infty)} (\rd r)]= \frac{\mu Z_0}{\beta}\, \ind_{(0, 1]}(r) \, \frac{\mathbf{g}(r)}{r} \, \rd r .
\]
The proof of this conjecture would require some uniform integrability conditions.

Unfortunately, due to the lack of a branching structure for the normalized process $\cz_t/\cz_0(\cti)$, it is not straightforward to obtain an expression for the mean measure of $\Phi_\mathrm{d}^{(\infty)}$ as in Theorem~\ref{the:meanSFS}.
\end{rem}

\begin{proof}
  Recall $-A$ denotes the  TMRCA and $N_t$ the number of ancestors at time $-t$
  of  the extant  population living  at time  $0$. The size at time 0 of the subpopulation descending from each 
  of these ancestors is distributed as an independent CB process $(Y_t,\ t\ge 0)$ conditioned to survive until 
  time $t$. Recall also the construction of the genealogical tree $\cti$ from
  Section~\ref{sec:def-cti}. We shall identify  $s\in \R$ with the
  element on the infinite spine $\R$ of $\cti$ at height $s$; in particular $\cti(s)$ represents the total offspring sub-tree of $s\in \R\subset \cti$. 

  Using  the branching
  property, we get for $h$ a non-negative measurable function defined on
  $[0, +\infty )$ with $h(0)=0$:
  \begin{align*}
    \Lambda(h)
    &= \E\left[\sum_{i\in I_M} h(\cz_0(\cti(m_i))) \right]\\
    &= \E\left[\int_\cti \cm(\rd x) \, h (\cz_0(\cti(x))) \right]\\
    &=  \mu  \E\left[ \int_0^\infty \rd t\, N_t\,  \rN[h(Y_t)\, |\, \zeta >
      t] \right]+
      \mu  \E\left[ \int_0^{A} \rd t\,  
 \E[h(\cz_0(\cti(-t)))\, |\, A>t] \right].
  \end{align*}
 In this formula, the first term represents the contributions at time 0 of 
 the \(N_t\) individuals at time \(t\) before the present that are ancestral 
 to the population at time 0, whereas the second term is the contribution of 
 the infinite spine, that is, the descendants at time 0 of populations 
 immigrating between time \(-t\) and 0. 

 The density  distribution $q_t$  of $Y_t$  under the  excursion measure
 $\rN$,     see    \cite[p.~63]{Li2011},     is    given     by: 
\[
  \rN[\rd Y_t=r]=q_t(r)\, \rd r= \frac{4\theta^2\expp{-2\beta\theta t}}
 {(1-\expp{-2\beta\theta t})^2} 
 \exp\left(-\frac{2\theta r}{1-\expp{-2\beta\theta t}}\right) \, \rd r.
\]
Using also the  expectation  of  $N_t$ 
 in~\eqref{eq:ENt} and     $\rN[\zeta>t]=c(t)$, we obtain for  the
 first  term that:
\begin{align*}
   \E\left[ \int_0^\infty \rd t\, N_t\,  \rN[h(Y_t)\, |\, \zeta >
  t] \right]
  &= \inv{\theta}  \int_{(0, \infty )} h(r)\, \rd r\,\int_0^\infty \rd t\,
    q_t(r) \\
  & = \int_{(0, \infty )} h(r) \,  \frac{\expp{-2 \theta
    r}}{\beta\theta r} \,  \rd r.
\end{align*}

We now consider the second term. Recall the stationary tree $\cti$ is obtained by grafting on the spine $\R$ at height $h_i$ a tree $\ct_i$, where $\cn (\rd h, \rd \bt)=\sum_{i\in I_T} \delta_{(h_i, \ct_i)} (\rd h, \rd \bt)$ is a Poisson point measure with intensity $2 \beta\, \rd h \, \N[\rd \ct]$. Recall that $\zeta_i$ denotes the height of $\ct_i$. Let $i_0\in I$ be the unique index such that $h_{i_0}=-A= \inf\{h_i\, \colon\, h_i+ \zeta_i>0 \text{ for }i\in I \}$. Notice that $h_{i_0}$ is a stopping time with respect to the filtration $(\cf_h, h\in \R)$, where $\cf_h$ is generated by $\cn((-\infty, h]\times \T)$. 
By the strong Markov property and homogeneity of the Poisson point measure, we get that the random measure $\sum_{i\in I_T} \delta_{(h_i- h_{i_0}, \ct_i)} (\rd h, \rd \bt)\ind_{\{h_i> h_{i_0}\}}$ is a Poisson point measure on $\R_+ \times \T$ with intensity $2 \beta\, \ind_{\{h>0\}}\, \rd h \, \N[\rd \ct]$.  We then deduce, thanks to the definition of the Kesten tree  given in Section~\ref{sec:def-cti},  that:
\begin{equation}\label{eq:MomentsZCl}
\E[h(\cz_0(\cti(-t)))\, |\, A>t] =\E[h(Z^\mathrm{Kesten}_t)].
\end{equation}
Thanks to~\eqref{eq:EZK}, we get that $
  \E[h(\cz_0(\cti(-t)))] =  \expp{2\beta
    \theta t}\rN[ Y_t \, h(Y_t)]$. 
Using this formula for the first equality and 
for the second  Equation~\eqref{eq:LoiASachantZ0} and that $Z_0$ is distributed as the
sum of two independent exponential random variables with parameter $2\theta$, we get:
\begin{align*}
   \E\left[ \int_0^{A} \rd t\,  
  \E[h(\cz_0(\cti(-t)))\, |\, A>t] \right]
  &= \int_0^\infty  \rd t\,  \P(A>t) \expp{2\beta
    \theta t}\rN[ Y_t \, h(Y_t)]\\
  &=  \int_0^\infty \rd t  \int_0^\infty \rd r \,
    (1-(1-\expp{-2\beta\theta t})^2) \expp{2\beta\theta t}  
   r q_t(r) h(r) \\
  & = \frac{2\theta}{\beta}  \int_{(0, \infty )} h(r) \rd r \,
 r \int_0^1 \frac{1+u}{u^2} 
 \expp{-2\theta r/u}  \rd u\\
 & = \frac{1}{\beta}\int_0^\infty \left( \expp{-2\theta r} + 
 2\theta r\int_0^1  \expp{-2\theta r /u} \frac{\rd u}{u} 
 \right)h(r) \rd r \\
 & = \frac{1}{\beta}\int_0^\infty \left(\expp{-2\theta r}
 +2\theta r\,\Gamma(0, 2\theta r)\right)h(r) \rd r.
\end{align*}
\end{proof}

\subsection{The clonal subpopulation size}\label{sec:PopClonale}
In this section, we consider the size $\zc$  of the 
clonal population at time 0, meaning the individuals sharing the same type as the 
MRCA $\root$ of the extant population at time 0:
\[
 \zc=\cz_0( \ctcl(\root)),
\]
with the clonal sub-tree defined by~\eqref{eq:clonal-tree}. 
Of course, we 
have $\zc\leq Z_0$ a.s.
Computing the expectation of $\zc^n$ (conditionally on $Z_0$) amounts to sampling $n$ independent vertices, say $\cx_1,  \ldots ,\cx_n$, in the stationary tree according to the measure $\cz_0$, such that the subtree  
$\ct_n'$
spanned by  the $n$ leaves  $\cx_1, \ldots , \cx_n$ and  the MRCA,  say $\root'$, has no mutation. By definition of the mutation point measure
$\cm$, conditionally on $\ct'_n$, this happens with probability $\exp{(- 
\mu  L(\cx_1,\ldots,  \cx_n))}$, where $ L(\cx_1,\ldots,  \cx_n)$ is  the  length of  the tree  $\ct_n'$. 
Thus,  we get   that for all $n\ge 1$:
\begin{equation}\label{eq:momentsZcl}
 \E[\zc^n\, |\, Z_0]
 =\E\left[\left.\int_{(\cti)^n} \expp{-\mu  L(\cx_1,  \ldots,
 \cx_n)} \, \prod_{i=1}^n \cz_0(\rd \cx_i) \, \right|\, Z_0\right].   
\end{equation}

  We recall the tree $\ct_n$ spanned by the $n$ leaves
$\cx_1, \ldots, \cx_n$ is rooted at  the MRCA of $\cx_1, \ldots, \cx_n$,
and is  thus a sub-tree of  $\ct'_n$ obtained by removing  the (possibly
empty) branch from $\root'$ to just before the MRCA of $\cx_1, \ldots, \cx_n$.

Following Section~\ref{sec:sample-T}, we consider the tree $\fT'_n$
defined as $\fT_n$ but for the last step where we cut the  semi-infinite
branch not at its
last (going downwards) branching point $\root_n$, which is at length
$\max_{1\leq 
  k\leq n} \zeta_k$, but
at $\root$ which is at length
$\max_{0\leq 
  k\leq n+1} \zeta_k$.
Notice that the distribution of $\max_{0\leq 
  k\leq n+1} \zeta_k$ does not depend on $n$,
see~\eqref{eq:LoiASachantZ0}, which explains why we do not stress the
dependence of $\root$ on $n$. See Fig.~\ref{fig:tikz} for an instance of
$\fT'_n$. 
Similarly to Lemma~\ref{lem:sampling}, using 
\cite[Lemma~4.1]{abrahamExactSimulationGenealogical2021}, we get the
following result.

\begin{lem}[Representation of the genealogical tree of $n$ individuals
  and the MRCA of the extant population]
 \label{lem:sampling2}
 For $n\in \N^*$, the rooted tree $\ct'_n$ spanned by
 \(\root'_n,\cx_1, \ldots, \cx_n\) is distributed as the rooted tree
 \(\fT' _n\).
\end{lem}

We thus deduce from \eqref{eq:momentsZcl} that:
\begin{equation}
   \label{eq:moment-Zcl}
  \E[\zc^n\, |\, Z_0]
  =\E\left[\left. Z_0^n \expp{-\mu  L_n }\, \right|\, Z_0\right]
\end{equation}
with $L_n$  the total
length of the tree $\fT'_n$. By construction the total length of
$\fT'_n$ is given by the length of the segments attached to the random
points $X_1, \ldots, X_{n-1}$ and the 
semi-infinite spine cut at $\max_{0\leq 
  k\leq n+1} \zeta_k$ which is attached to $X_0=0$, that is:
\[
  L_n= \max_{0\leq k\leq n+1} \zeta_{k} + \Lambda_{n-1}
  \quad\text{and}\quad
  \Lambda_{n-1} =\sum_{k=1}^{n}
   \zeta_{k}.   
 \]
 (Notice that in  the above formula $\zeta_\ell=0$ for  the index
 $\ell\in \llbracket 1, n \rrbracket$
 such  that $X_{(\ell)}=X_0$.)

\begin{rem}[On the asymptotic of $L_n$]
   \label{rem:asympt-Ln}
 Let us  mention that  the asymptotics  of
   $\Lambda_{n-1}$ have been computed
   in \cite[Section~5]{abrahamExactSimulationGenealogical2021}, and  we have the following convergence in distribution:
   \[
\Lambda_{n-1} - \frac{Z_0}{\beta} \log\left(\frac{n}{2 \theta
    Z_0}\right) \xrightarrow[n\rightarrow \infty ]{\text{(d)}} \cll,
\]
where      the      distribution       of      $\cll$      is      given
in~\cite[Lemma~5.4]{biTotalLengthGenealogical2016} (with  $\cll$ denoted
by $W_0$ therein).   In fact the construction of the  $\zeta_k$'s can be
done   in    such   way   that    this   convergence   is    a.s.,   see
\cite[Theorem~5.1]{abrahamExactSimulationGenealogical2021}. 
However, we did  not investigate  the   joint  law of $\cll$ and the TMRCA of the whole population at time 0 given by 
$\max_{0\leq   k\leq   n+1}  \zeta_{k}$ (which we recall does not depend on $n$).  
\end{rem}

\medskip

Recall that $\beta(a,b)=\Gamma(a)\Gamma(b)/\Gamma(a+b)$ for $a,b\in \R_+^*$. 
We  set:
\[
 R= \frac{\zc}{Z_0}
\quad\text{and}\quad
\alpha=\frac{\mu}{2\beta \theta}\cdot
\]

\begin{theo}
 \label{th:moment-clone}
 For $n\in \N^*$, we
 have:
\begin{equation}
 \label{eq:E[ZR]-prop}
 \E[\zc^{n-1} R]
= \frac{\alpha}{(1+\alpha)^n}
 \left[ \beta\left(n, \frac{2+\alpha}{1+\alpha}\right) + \beta\left(n, \frac{\alpha}{1+\alpha}\right)
 - \frac{2}{n}\right]\, \, \E\left[Z_0^{n-1}\right]. 
\end{equation}
The formula for $\E[\zc^{n}]$
is explicit and given by~\eqref{eq:trop-long} below. 
In particular, we have:
\begin{equation}
 \label{eq:Z-R}
 \E[\zc]=\E[R Z_0]= \frac{6}{(\alpha+1)(\alpha+2)(\alpha+3)}\, \E[Z_0]
 \quad\text{and}\quad
 \E[R]=\frac{2}{(\alpha+1)(\alpha+2)} 
\end{equation}
and thus:
\[
 \cov(R, Z_0)= - \frac{2\alpha}{(\alpha+1)(\alpha+2)(\alpha+3)} \, 
\E[Z_0].
 \]
We also have:
\[
 \frac{\E[\zc^n]}{ \E[Z_0^n]}
 \sim_{n\rightarrow \infty } \frac{2\alpha}{2+\alpha}
\Gamma\left(\frac{\alpha}{1+\alpha}\right) \frac{1
}{(1+\alpha)^n \, n^{\alpha/(1+\alpha)}} \cdot
\]
\end{theo}

Interestingly, Theorem \ref{th:moment-clone} shows that $R$ and $Z_0$ are negatively correlated as 
$\mathrm{Corr}(R, Z_0)= -1 + 3/(\alpha+3)$
: larger populations tend to have smaller clonal subpopulations, and this effect becomes stronger as the mutation rate increases.

\subsection{Proof of Theorem~\ref{th:moment-clone}}

We shall use many times the following formula for $b>0$:
\[
 \beta(1,b)=\inv{b}
 \quad\text{and}\quad
 \beta(a,b)\sim_{a\rightarrow \infty }
 \Gamma(b)\, a^{-b},
\]
 and, as $\Gamma(b+1)=b\Gamma(b)$, for $a>1$:
\[
 \beta(a-1, b+1)=\frac{b}{a-1} \, \beta(a,b) .
\]
We shall also use that for $U$ uniform on $[0, 1]$, $a\geq 0$, $k> 0$ and $b=a/(1+\alpha)$:
\begin{equation}
 \label{eq:E[U]}
\cu(k,a):= \E\left[ U^{\alpha+a} \, \left(1-
 U^{1+\alpha} \right)^{k-1} \right]=
 \inv{1+\alpha}\, \beta(k, 1+b),
\end{equation}
 that for $a>0$ and $k>1$:
\begin{equation}
 \label{eq:E[U]-1}
\cu(k-1,a)=
 \frac{a}{k-1} \, \inv{(1+\alpha)^2}\, \beta(k, b),
\end{equation}
and that for $a>1+\alpha$ and $k>2$:
\begin{equation}
 \label{eq:E[U]-2}
\cu(k-2,a)=
 \frac{a(a-1-\alpha)}{(k-1)(k-2)} \, \inv{(1+\alpha)^3}\, \beta(k, b-1).
\end{equation}

 Let $n\in \N^*$. 
As $Z_0$ is the sum of two independent exponential random variables with
mean $1/2\theta$, we get:
\[
 \E[Z_0^{n-1}]=\frac{n!}{(2\theta)^{n-1}}\cdot
\]
Using~\eqref{eq:moment-Zcl}, we first consider the quantity:
 \[
 \E[\zc^{n-1} R]=\E\left[Z_0^{n-1} \expp{- \mu  L_n}\right].
\]

\medskip

We shall now go back to the definition of the random variables
$(\zeta_k, 0\leq  k\leq  n+1)$ from Item (ii) of
Section~\ref{sec:sample-T} to give a nice representation of the
distribution of $(\max_{0\leq k\leq n+1} \zeta_{k} , \Lambda_{n-1})$
under the probability measure $\rd \Q_n= Z_0^{n-1} \rd \P/
\E[Z_0^{n-1}]$. Thus,  since $2\theta Z_0$ has  the  $\Gamma(2,1)$
distribution, we obtain that under $\Q_n$ it has the $\Gamma(n+1,1)$ distribution.

\medskip
Recall that the random variables $X_{(0)}=-\Eg<X_{(1)}< \ldots <X_{(n)}< X_{(n+1)}=\Ed$
 are the order statistics of $\{-\Eg,\Ed,X_0,\dots,X_{n-1}\}$ where $X_0=0$ and \(X_k = Z_0 U_k - \Eg\)
for \(k\in \N^*\) and where \((U_k,\ k\in \N^*)\) are independent 
random variables, uniformly distributed on \([0,1]\), independent of \(\Eg,\Ed\). Recall $Z_0=\Eg+\Ed$. 
 
In           particular          the           random          variables
$(\Delta_k=2\theta(X_{(k)}   -X_{(k-1)}),   1\leq    k\leq   n+1)$   are
distributed as $(2\theta Z_0 (  U_{(k)} - U_{(k-1)}), 1\leq k\leq n+1)$,
where $U_{(0)}=0  < U_{(1)}  < \ldots  < U_{(n)}  < U_{(n+1)}=1$ are the
order  statistics of  $\{0, 1,  U_1,  \ldots, U_n\}$.  Using properties of  the Poisson
process,   we   deduce  that   under   $\Q_n$,   the  random   variables
$(\Delta_k, 1\leq k\leq n+1)$ are distributed as  $(E_k, 1\leq k\leq
n+1)$, where $\mathbf{E}=(E_k, k\in \N^*)$ are independent exponential random
variables with mean 1.

\medskip

Set $(\zeta'_k, 1\leq k\leq n+1)$ with:
 \[
 \zeta'_k= \inv{2\beta \theta}\,
 \log\left(\frac{E'_k+E_k}{E'_k}\right),
 \]
 where  the random  variables  $(E'_k, k\in  \N^*)$  are distributed  as
 $\mathbf{E}$ and independent of  $\mathbf{E}$.  Now recall there exists
 a   (random)  index   $i\in  \llbracket   1,n  \rrbracket$   such  that
 $\zeta_i=0$,   so  intuitively   among   the   $n+2$  random   variables
 $\zeta_0, \ldots, \zeta_{n+1}$, there are  only $n+1$ nontrivial ones.
 More            precisely,            we            get            that
 $\left(\max_{0\leq  k\leq  n+1}  \zeta_{k} ,  \Lambda_{n-1}\right)$  is
 under $\Q_n$ distributed as:
\[
  \left(\max_{1\leq k\leq n+1} \zeta'_{k} ,\,  \sum_{k=2}^{n}
   \zeta'_{k}   
 \right).
 \]

 The random variables $(V_k, k\in \N)$, with:
 \[
 V_k  =\frac{ E'_k}{E'_k+ E_k},
 \]
 are independent
 and uniformly distributed on \([0,1]\).  We deduce that:
 \begin{equation}
 \label{eq:E[ZR]}
\E[\zc^{n-1} R]=\E\left[Z_0^{n-1}\right]\, \E\left[\left(\min_{1\leq k\leq n+1}
 V_k^\alpha\right) \, \prod_{j=2}^{n} V_j^\alpha\right].
 \end{equation}
Elementary computations give that:
\[
  \E\left[\left(\min_{1\leq k\leq n+1}
      V_k^\alpha\right) \, \prod_{j=2}^{n} V_j^\alpha\right]
  = 2A_n+ (n-1) B_n,
 \]
 with, thanks to~\eqref{eq:E[U]}:
 \begin{align*}
 A_n
 =\E\left[V_1^\alpha \, \prod_{k=2}^{n+1} \ind_{\{V_1< V_k\}}\,
 \prod_{j=2}^{n} V_j^\alpha \right]
& = \frac{1}{(1+\alpha)^{n-1}}\, \E\left[ U^\alpha (1-U) \, \left(1-
 U^{1+\alpha} \right)^{n-1} \right]\\
& = \frac{1}{(1+\alpha)^{n-1}}\, \Big(\cu(n, 0) - \cu(n, 1) \Big)\\
&=\inv{(1+\alpha)^n} \left[ \inv{n} - \beta\left(n,
 \frac{2+\alpha}{1+\alpha} \right)\right],
 \end{align*}
and for $n\geq 2$, thanks to~\eqref{eq:E[U]-1}:
 \begin{align*}
 B_n
 &= \E\left[V_2^{2\alpha} \prod_{k=1,3, \ldots, n+1} 
 \ind_{\{V_2< V_k\}}\, \prod_{j=3}^{n-1} V_j^\alpha\right]\\
 &= \inv{ (1+\alpha)^{n-2}} \, \E\left[ U^{2\alpha} (1-U)^2 \, \left(1-
 U^{1+\alpha} \right)^{n-2} \right]\\
 &= \inv{ (1+\alpha)^{n-2}} \, \Big(\cu(n-1,\alpha) -
 2\cu(n-1,1+\alpha) + \cu(n-1, 2+\alpha) \Big)\\
 &= \inv{n-1}\, \inv{(1+\alpha)^{n}} \left[ 
 \alpha \beta \left(n, \frac{\alpha}{1+\alpha}\right)
 -\frac{2(1+\alpha)}{n} 
+ (2+\alpha) \beta \left(n, \frac{2+\alpha}{1+\alpha}\right)
 \right].
 \end{align*}
 We deduce that:
 \[
 2A_n + (n-1) B_n=\frac{\alpha}{(1+\alpha)^n}
 \left[ \beta\left(n, \frac{2+\alpha}{1+\alpha}\right) + \beta\left(n, \frac{\alpha}{1+\alpha}\right)
 - \frac{2}{n}\right]. 
\]
We thus deduce~\eqref{eq:E[ZR]-prop} from~\eqref{eq:E[ZR]}. Taking
$n=1$, gives the value of $\E[R]$ in~\eqref{eq:Z-R}.

\medskip

 We now compute $\E[\zc^n]$. We have:
 \begin{align*}
 \E[\zc^n]
 &= \inv{2\theta} \E\left[Z_0^{n-1} (2\theta Z_0)\, \expp{- \mu
L_n}\right] \\ 
 &=\inv{2\theta} \, \E[Z_0^{n-1}] \, 
 \E\left[(E_1+ \ldots + E_{n+1})\, \min_{1\leq k\leq n+1}
 \left(\frac{E'_k}{E_k+E'_k}\right)^\alpha \, \, 
 \prod_{j=2}^{n}\left(\frac{E'_j}{E_j+E'_j}\right)^\alpha\right]\\
 &= \frac{1}{n+1} \, \E[Z_0^{n}] \, \Big(2 C_n+(n-1)D_n\Big),
 \end{align*}
 with:
 \[
 C_n
 =\E\left[E_1\, \min_{1\leq k\leq n+1}
 \left(\frac{E'_k}{E_k+E'_k}\right)^\alpha \, 
 \prod_{j=2}^{n}\left(\frac{E'_j}{E_j+E'_j}\right)^\alpha\right]
 \]
 and for $n\geq 2$:
 \[
   D_n
   =\E\left[E_2\, \min_{1\leq k\leq n+1}
 \left(\frac{E'_k}{E_k+E'_k}\right)^\alpha \, 
 \prod_{j=2}^{n}\left(\frac{E'_j}{E_j+E'_j}\right)^\alpha\right].
 \]
 We have:
 \begin{align*}
 C_n
 &=\E\left[(E_1+ E'_1) \, \left(1- V_1\right)\, 
\left( \min_{1\leq k\leq n+1} 
V_k^\alpha\right) \,
 \prod_{j=2}^{n} V_k ^\alpha\right]\\
 &=2\E\left[ \left(1- V_1\right)\, 
\left( \min_{1\leq k\leq n+1} 
V_k^\alpha\right) \,
 \prod_{j=2}^{n} V_k ^\alpha\right]\\
 &=2A^{(0)}_n+ A^{(01)}_n+ (n-1)B^{(0)}_n,
 \end{align*}
 where we used that $E_1+ E'_1$ is independent of $V_1$ for the second
 equality, and with:
\begin{align*}
 A^{(0)}_n
 &=\E\left[ (1-V_1) V_1^\alpha \, \prod_{k=2}^{n+1} \ind_{\{V_1< V_k\}} \,
 \prod_{j=2}^{n}V_j^\alpha\right]\\
&=\inv{(1+\alpha)^{n-1}}\, \E\left[U^\alpha (1-U)^2\, (1-
 U^{\alpha+1})^{n-1}\right]\\
 &=\inv{(1+\alpha)^{n-1}}\,\Big(
 \cu(n,0) - 2 \cu(n, 1) + \cu(n,2)\Big)\\
& = \inv{(1+\alpha)^n} \,\left[\inv{n} -
 2\beta\left(n,\frac{2+\alpha}{1+\alpha} \right) +
 \beta\left(n, \frac{3+\alpha}{1+\alpha}\right) \right] ,
 \end{align*} 
and (using elementary computations for the last equality):
 \begin{align*}
 A^{(01)}_n
 =2\E\left[ (1-V_1)\, V_{n+1}^\alpha \, \prod_{k=1}^{n} \ind_{\{V_{n+1}< V_k\}} \,
 \prod_{j=2}^{n}V_j^\alpha\right]
= A^{(0)}_n,
 \end{align*} 
and for $n\geq 2$:
 \begin{align*}
 B^{(0)}_n
 &= 2\E\left[ (1-V_1)\, V_2^{2\alpha} \, \prod_{k=1,3,  \ldots, n+1}
 \ind_{\{V_2< V_k\}} \, 
 \prod_{j=3}^{n}V_j^\alpha\right]\\
 &=\inv{(1+\alpha)^{n-2}} \, \E\left[U^{2\alpha} (1-U)^3 (1-
 U^{1+\alpha})^{n-2}\right]\\
 &=\inv{(1+\alpha)^{n-2}} \,
 \Big( \cu(n-1, \alpha) - 3\cu(n-1, 1+\alpha) + 3 \cu(n-1, 2+\alpha)-
 \cu(n-1, 3+\alpha)\Big)\\
 &=\inv{n-1}\, \inv{(1+\alpha)^{n}} \,
 \left[ \alpha \beta\left(n, \frac{\alpha}{1+\alpha} \right)
 - 3\frac{(1+\alpha)}{n} \right. \\
 & \left. \qquad\qquad\qquad\qquad\qquad\qquad + 3(2+\alpha) \beta\left(n, \frac{2+\alpha}{1+\alpha} \right)
- (3+\alpha) \beta\left(n, \frac{3+\alpha}{1+\alpha} \right)\right].
 \end{align*}
Similarly, we also have for $n\geq 2$:
 \begin{align*}
 D_n
 &=\E\left[(E_2+ E'_2) \, \left(1-V_2\right)\, 
 \left( \min_{1\leq k\leq n+1} 
 V_k^\alpha \right)\,
 \prod_{j=2}^{n} V_j ^\alpha\right]\\
 &=2 \E\left[ (1-V_2)\, \left(\min_{1\leq k\leq n+1}
 V_k\right)^\alpha \,
 \prod_{j=2}^{n}V_j^\alpha\right]\\
 &=2\left(2A^{(1)}_n+ B^{(11)}_n+(n-2)B^{(1)}_n\right),
 \end{align*}
 with:
\[
 A^{(1)}_n
=\E\left[ V_1^{\alpha}(1-V_2) \, \prod_{k=2}^{n+1} \ind_{\{V_1< V_k\}} \,
 \prod_{j=2}^{n}V_j^\alpha\right]= A_n - A_n^{(2)},
\]
and:
\begin{align*}
 A^{(2)}_n
 &=\E\left[ V_1^{\alpha}V_2 \, \prod_{k=2}^{n+1} \ind_{\{V_1< V_k\}} \,
 \prod_{j=2}^{n}V_j^\alpha\right]\\
 &= \inv{(2+\alpha)\, (1+\alpha)^{n-2}}\, \E\left[U^\alpha(1-U) (1-
 U^{2+\alpha})(1-
 U^{1+\alpha})^{n-2} \right]\\
 &=\inv{(2+\alpha)\, (1+\alpha)^{n-2}}\,
 \Big(
 \cu(n-1,0) - \cu(n-1, 2+\alpha) - \cu(n-1, 1) + \cu(n-1, 3+\alpha)
 \Big)\\
 &=\inv{n-1}\, \inv{(2+\alpha)\, (1+\alpha)^{n}}\,
 \left[
 (1+\alpha)
 - (2+\alpha) \beta\left(n,\frac{2+\alpha}{1+\alpha}\right) \right. \\
 & \left. \qquad\qquad\qquad\qquad\qquad\qquad - \beta\left(n,\frac{1}{1+\alpha}\right)
+(3+\alpha) \beta\left(n,\frac{3+\alpha}{1+\alpha}\right)\right],
\end{align*}
and with (using elementary computations for the last equality):
\[
 B^{(11)}_n
 = \E\left[ (1-V_2)\, V_2^{2\alpha} \, \prod_{k=1, 3, \ldots, n+1}
 \ind_{\{V_2< V_k\}} \, 
 \prod_{j=3}^{n}V_j^\alpha\right]
= B^{(0)}_n,
\]
 and lastly with, for $n\geq 3$:
\begin{align*}
 B^{(1)}_n
 &= \E\left[ (1-V_2)\, V_2^\alpha\, V_3^{2\alpha} \, \prod_{k=1,2, 4, \ldots, n+1}
 \ind_{\{V_3< V_k\}} \, 
 \prod_{j=4}^{n}V_j^\alpha\right]\\
 &= B_n - B^{(2)}_n,
 \end{align*} 
 and, using~\eqref{eq:E[U]-2}:
 \begin{align*}
 B^{(2)}_n
 &=\inv{(2+\alpha)\, (1+\alpha)^{n-3}}\,
 \E\left[U^{2\alpha} (1-U)^2
 (1-U^{2+\alpha}) (1-
 U^{1+\alpha})^{n-3}\right] \\
 &=\inv{(2+\alpha)\, (1+\alpha)^{n-3}}\,
 \Big(
 \cu(n-2,\alpha) - \cu(n-2, 2+2\alpha) \Big.
 \\
 &\hspace{2cm} \Big. - 2\cu(n-2, 1+\alpha) +2\cu(n-2 ,
 3+2\alpha) + \cu(n-2, 2+\alpha) - \cu(n-2, 4+2\alpha)
 \Big)\\
 &=\inv{(n-1)(n-2)}\, \inv{(2+\alpha)\, (1+\alpha)^{n}}\,
 \Big[
 (n-1) \alpha(1+\alpha) \beta\left(n-1,\frac{\alpha}{1+\alpha}\right)
 -\frac{(2+2\alpha)(1+\alpha)}{n} 
 \Big.
 \\
 &\hspace{4cm} 
 - 2(1+\alpha)^2
 +2 (3+2\alpha)(2+\alpha) \beta\left(n,\frac{2+\alpha}{1+\alpha}\right)\\
 &\hspace{4cm} \Big.
 + (2+\alpha)\beta\left(n,\frac{1}{1+\alpha}\right)
 - (4+2\alpha)(3+\alpha)\beta\left(n,\frac{3+\alpha}{1+\alpha}\right)
 \Big].
 \end{align*}

In conclusion, we obtain that:
\begin{multline}
 \label{eq:trop-long}
 \E[\zc^n]\\ 
\begin{aligned}
 &= \frac{1}{n+1} \, \E[Z_0^{n}] \, \Big(2 C_n+(n-1)D_n\Big)\\
 &=\frac{2}{n+1} \, \E[Z_0^{n}] \, \Big(3 A^{(0)}_n +2(n-1)(A_n
 -A_n^{(2)} +B_n^{(0)} ) + (n-1)(n-2) ( B_n 
 -B^{(2)}_n) \Big).
\end{aligned} 
\end{multline}
Taking $n=1$ in the above formula, we get:
\[
\E[\zc]= 3 A_1^{(0)}\, \E[Z_0] 
 = \frac{3}{(1+\alpha)} \,\left[1 - 
 2\frac{1+\alpha}{2+\alpha} + 
 \frac{1+\alpha}{3+\alpha}\right]\, \E[Z_0] ,
\]
 which gives the first part of~\eqref{eq:Z-R}. 
 We now give the leading term in~\eqref{eq:trop-long}. We have:
 \begin{align*}
 (1+\alpha)^{n}\, A_n^{(0)}
 &=O(n^{-1}),\\
 (1+\alpha)^{n}\, A_n
 &=O(n^{-1} ),\\
 (1+\alpha)^{n}\, A_n^{(2)}
 &= O(n^{-1} ),\\
 (1+\alpha)^{n}\, B_n^{(0)}
 &= o(n^{-1} ),\\
 (1+\alpha)^{n}\, B_n
 &=n^{-1 -\alpha/(1+\alpha)} \alpha \Gamma\left(\frac{\alpha}{1+\alpha}\right)+ O(n^{-2}
 ),\\
 (1+\alpha)^{n}\, B_n^{(2)}
 &= n^{-1 -\alpha/(1+\alpha)} \, \frac{1+\alpha}{2+\alpha} \, \alpha 
 \Gamma\left(\frac{\alpha}{1+\alpha}\right)+ O(n^{-2} ).
 \end{align*}
 We deduce that:
 \[
 \E[\zc^n]= \frac{2\alpha}{2+\alpha} \, 
 \Gamma\left(\frac{\alpha}{1+\alpha}\right) \E[Z_0^{n}] \,
 \inv{(1+\alpha)^n} \, \left( \inv{n^{\alpha/(1+\alpha)}}\, +
 O(n^{-1})\right).
 \]
 This ends the proof of Theorem~\ref{th:moment-clone}.

\end{document}